\documentclass[letter,11pt]{article}
\usepackage[top=1in, bottom=1in, left=1in, right=1in]{geometry}
\usepackage{amssymb}
\usepackage{amsthm}
\usepackage{amsfonts}
\usepackage{amsmath}
\usepackage{bbm}
\usepackage{abstract}
\usepackage{color}

\newtheorem{theorem}{Theorem}[section]
\newtheorem{lemma}[theorem]{Lemma}
\newtheorem{prop}[theorem]{Proposition}

\theoremstyle{definition}

\newtheorem{remark}[theorem]{Remark}

\newcommand{\Ric}{{\rm Ric}}

\newcommand{\ddbar}{\partial\bar\partial}

\usepackage{todonotes}

\usepackage{hyperref}
\hypersetup{
    bookmarks=true,         
    unicode=false,          
    pdftoolbar=true,       
    pdfmenubar=true,       
    pdffitwindow=false,    
    pdfstartview={FitH},      
    pdfauthor={T. Darvas, K.-R. Wu},     
    colorlinks=true,       % false: boxed links; true: colored links
   linkcolor=black,          % color of internal links
    citecolor=black,        % color of links to bibliography
    filecolor=black,      % color of file links
    urlcolor=black}           % color of external links

\title{Griffiths extremality, interpolation of norms, and K\"ahler quantization }
\author{Tam\'as Darvas and Kuang-Ru Wu}
\date{\vspace{-0.2in}}

\begin{document}
\maketitle
\begin{abstract} Following Kobayashi, we consider Griffiths negative complex Finsler bundles, naturally leading us to introduce Griffiths extremal Finsler metrics. As we point out, this notion is closely related to the theory of interpolation of norms, and is characterized by an equation of complex Monge--Amp\`ere type, whose corresponding Dirichlet problem we solve.  As applications, we prove that Griffiths extremal Finsler metrics quantize solutions to a natural PDE in K\"ahler geometry, related to the construction of flat maps for the Mabuchi metric.
\end{abstract}

\section{Introduction}

Let $\pi: E \to Y$ be a complex vector bundle of rank $r$ over a complex manifold $Y$ of dimension $m$. We say that $(E,f)$ is a (complex) \emph{Finsler bundle} with $f$ being the \emph{Finsler metric} if $f(y,\xi) \geq 0$ for all $y \in Y$ and $\xi \in E_y$, with $f(y,\xi) =0$ if and only if $\xi =0$. Moreover, $f(y, \lambda \xi) = |\lambda| f(y,\xi)$ for all $y \in Y,\xi \in E_y$ and $\lambda \in \Bbb C$. In case  $f(y, \cdot)$ also satisfies the triangle inequality on each $E_y$, then $f$ is a \emph{Finsler norm}. As usual, if $f(y, \cdot)$ is induced by a quadratic form, then $f$ is a Hermitian metric. We do not assume any regularity at this point. Later on, suitable regularity will be imposed on the Finsler metric $f$.

Despite their ubiquity, complex geometers often consider Finsler metrics too general to be relevant, with the special case of Hermitian metrics receiving most of the attention in the literature. Somewhat reversing this trend, the main purpose of the present article is to show that Finsler metrics arise naturally in the quantization of K\"ahler metrics that are solutions to canonical PDEs in the subject. Indeed, it seems that many times there are not enough Hermitian metrics to quantize such solutions.

Broadly speaking, we point out that interpolation methods for Euclidean norms going back to Rochberg \cite{Ro84},  Slodkowski \cite{Sl1,Sl2,Sl3}, Coifman--Semmes \cite{CS93}, and more recently Berndtsson--Cordero--Erausquin--Klartag--Rubinstein \cite{BCKR16} can be adapted to our Finsler setting, naturally leading us to Griffiths extremality. More importantly, we show that Griffiths extremal Finsler metrics quantize the solution to a natural complex Monge--Amp\`ere equation in K\"ahler geometry considered by Berman--Demailly and Darvas--Rubinstein \cite{BD09,DR16}, closely related to the notion of flatness in Mabuchi geometry. The link is provided by Berndtsson's theorem for positivity of direct images \cite{Brn09}, and classical methods of Kobayashi on Finsler metrics \cite{Ko75}.

\vspace{-0.1in}\paragraph{Griffiths negativity/extremality and interpolation of norms.} Before we address K\"ahler quantization, let us review Griffiths negativity of Finsler bundles and introduce Griffiths extermality, pointing out connections to interpolation of Euclidean norms at the end. 
Complex Finsler bundles were considered by Kobayashi \cite{Ko75}, who was motivated by the Griffiths conjecture on the relationship between ampleness and positivity of vector bundles \cite{Gr69} (for  recent progress see \cite{Brn09, MT07, FLW09}, and references within). For a detailed study of complex Finsler geometry we refer to the survey \cite{Wo07} and the book \cite{Ko14}, whose terminology we follow. 

In this paper, a Finsler bundle $\pi: (E,f) \to Y$ is said to be \emph{Griffiths negative} if $f$ is plurisubharmonic (psh) on the total space of the bundle $E$. With slight abuse of precision, this means that $f$ is upper semi-continuous (usc) and $i\ddbar f \geq 0$ on $E$, in the sense of currents. As is well known, in case $f$ is a smooth Hermitian metric this definition recovers the usual definition of Griffiths negativity \cite[Chapter VII]{De12}. For more details and comparison with Kobayashi's papers \cite{Ko75, Ko96}, see Section \ref{subsec griff neg of finsler}. 

One may ask, which are the least Griffiths negative Finsler metrics? As we are dealing with plurisubharmonicity, one is naturally led to the conditions  $i\ddbar f \geq 0$ together with $\dim \textup{Ker } i\ddbar f \geq 1$. This would be fine for smooth $f$, however in our case $f$ is only psh, hence it is more precise to ask that $f$ satisfies the complex Monge--Amp\`ere equation $(i\ddbar f)^{m+r} =0$ on $E$ in the sense of Bedford--Taylor \cite{BT76}. Griffiths negative metrics $f$ satisfying this equation will be called \emph{Griffiths extremal}.

We are naturally led to the following Dirichlet problem: given a relatively compact open subset $D \subset Y$ with smooth boundary and $g$ a Finsler metric on $E|_{\partial D}$, is it possible to find a Griffiths extremal metric $f$ on $E|_{\overline{D}}$ assuming the values of $g$ on the boundary? Our first result says that this is indeed the case under reasonable regularity assumptions on $D$ and $g$.

\begin{theorem}\label{thm: main PDE theorem} Let $D \subset Y$ be a relatively compact strongly pseudoconvex domain with smooth boundary, and $g$ a continuous Finsler metric on $E|_{\partial{D}}$ such that $g_z:=g(z,\cdot)$ is psh on $E_z, \ z \in \partial D$. Then there exists a unique continuous Finsler metric $f$ on $E|_{\overline{D}}$ solving the following Dirichlet problem:
\begin{equation}\label{eq: Griff_extr_equation_intr}
    \begin{cases}
    (i\partial \bar{\partial} f)^{m+r}=0 \text{  on $E|_D$,}\\
    f \in \textup{PSH}(E|_D),\\
    f=g, \text{  on $E|_{\partial D}$}.
        \end{cases}
\end{equation}
\end{theorem}

By continuity of the Finsler metrics $f$ and $g$ we understand continuity on the total space of the bundles $E|_{\overline{D}}$ and $E|_{\partial D}$ respectively.
Strong pseudoconvexity means that there exists an open set $D'\subset Y$ such that $ D' \supset \overline{D}$ and $\rho: D' \to \Bbb R$ smooth and strongly psh defining function of $\overline{D}$ (i.e., $d\rho\neq 0$ on $\partial D$, $\partial D = \{\rho = 0 \}$, and $D = \{\rho <  0\}$). 
This is perhaps the most natural condition to pose for $D$, when trying to solve complex Monge--Amp\`ere equations. Moreover, since $E|_{\partial D}$ contains complex directions, it is \emph{necessary} to ask that $g$ is psh in these directions, otherwise no Griffiths negative metric could assume these values. Interestingly enough, since $D$ is strongly pseudoconvex we need not require plurisubharmonicity of $g$ in directions of $\partial D$.

Similar equations have been already solved in the literature, the difficulty here is the fact the underlying domain $E|_{\overline{D}}$ is non-compact. We will circumvent this issue by a convenient projectivization argument. Moreover, as in the case of compact domains, we will see that the solution $f$ arises as the solution to the following Perron process:
\begin{equation}\label{eq: Griff_Extr_Metric_Perron}
f := \sup_{h \in F^\mathcal M_g} h,
\end{equation}
where $F^\mathcal M_g$ is the following family of Griffiths negative metrics:
$$ F^\mathcal M_g := \{h \textup{ is a Griff. negative Finsler metric on }E|_{D}, \limsup_{E|_{D} \ni (x',\xi') \to (x,\xi) \in E|_{\partial D}} h(x',\xi') \leq g(x,\xi)\}.$$
Staying with Perron processes, when $D \subset \Bbb C^m$ is a domain and $E$ is trivial, the envelope of $ F^\mathcal M_g $ has been considered specifically by Slodkowski in \cite{Sl1}. Slodkowski interpreted $f$ as the interpolation of the Finsler norm $g$ from $E|_{\partial D}$ into $E|_{\overline{D}}$. Interestingly, the connection with complex Monge--Amp\`ere equations seems to have not been emphasized, and the extension to pseudoconvex manifolds  seems not to have been considered in the literature until now either. With different motivation, the case $m=1$ and $E$ trivial has been considered by Coifman--Cwikel--Rochberg--Sagher--Weiss \cite{CCRSW82} and Rochberg \cite{Ro84}. Higher dimensional cases have been considered by Slodkowski \cite{Sl2,Sl3,Sl4}, Coifman--Semmes \cite{CS93} and Berndtsson--Cordero-Erausquin--Klartag--Rubinstein \cite{BCKR16}, with connections to Yang--Mills equations pointed out by Donaldson \cite{Do92}. 

In case the boundary data $g$ is a Hermitian metric (or a Finsler norm) one might naturally wonder if the solution $f$ to \eqref{eq: Griff_extr_equation_intr} is also a Hermitian metric (or a Finsler norm) on $D$. As pointed out by Slodkowski in \cite[Corollary 6.8]{Sl2} this may not be the case (in fact, he gave an example where the boundary data $g$ is a Finsler norm, but the envelope $f$ is not). However, in the Finsler norm case one may still consider the following partial Perron envelope 
\begin{equation}\label{eq: Griff_Extr_Norm_Perron}
\tilde f := \sup_{h \in F^\mathcal N_g} h,
\end{equation}
where $F^\mathcal N_g$ is the following family of Griffiths negative norms(!):
$$ F^\mathcal N_g := \{h \textup{ is a Griff. negative Finsler norm on }E|_{D}, \limsup_{E|_{D} \ni (x',\xi') \to (x,\xi) \in E|_{\partial D}} h(x',\xi') \leq g(x,\xi)\}.$$
Naturally $\tilde f \leq f$.  As shown by Slodkowski in \cite{Sl1}, in case $D$ is a (strongly pseudoconvex) smooth subdomain of $\Bbb C^m$ and $E$ is trivial, $\tilde f$ is a continuous Griffiths negative Finsler norm on $E|_D$ that assumes the values of $g$ on $E|_{\partial D}$. On the other hand, according to Slodkowski's example \cite[Corollary 6.8]{Sl2}, $f$ is not always a Finsler norm, so $f \neq \tilde f$ and $\tilde{f}$ does not solve \eqref{eq: Griff_extr_equation_intr} in general by uniqueness of the solution.

\vspace{-0.1in}\paragraph{A degenerate complex Monge--Amp\`ere equation in K\"ahler geometry.} Now we introduce a natural Dirichlet problem in K\"ahler geometry, closely related to Mabuchi flatness (see Section 2), whose solution we will quantize via the Griffiths extremal metrics of the previous paragraph.
Let $(X,L)$ be an $n$-dimensional compact K\"ahler manifold with an ample Hermitian line bundle $(L,h)$ such that $\omega:= i\Theta(h) >0$ is K\"ahler. Given this data, one can introduce the space of K\"ahler potentials and $\omega$-psh potentials respectively:
$$\mathcal H_\omega := \{u \in C^\infty(X) \textup{\ s.t. \ } \omega + i\ddbar u >0\}.$$
$$\textup{PSH}(X,\omega) := \{u:X \to [-\infty,\infty), u \textup{ is usc and } \omega + i\ddbar u \geq 0\}.$$

As before, $D \subset Y$ is a relatively compact strongly pseudoconvex domain with $\dim Y =m$. Moreover, given a boundary data $v \in C(\partial{D} \times X)$ satisfying $v_z:=v(z,\cdot) \in \textup{PSH}(X,\omega), \ z \in \partial D$, we consider the following equation: for a function $u$ on $\overline{D}\times X$
\begin{equation}
    \begin{cases}\label{eq: DR_equation}
    (\pi^*\omega +i\partial \bar{\partial}u)^{n+m}=0 \text{  on $D\times X$},\\
    u|_{\partial D\times X}=v\\
    \pi^*\omega +i\partial \bar{\partial}u\geq 0,
    \end{cases}
\end{equation} 
where $\pi: Y \times X \to X$. This equation was perhaps first considered by Berman--Demailly \cite{BD09}, when the boundary data is $C^2$ (see also \cite{Bo12,Bl13} for related cases). When $D \subset \Bbb C^m$ is a Bochner tube, this equation was studied in \cite{DR16}. In this latter case, due to the symmetry in the imaginary directions, it was shown in \cite{DR16} that the solution to \eqref{eq: DR_equation} is the Legendre transform of a family of rooftop envelopes.

When $D$ is a  Riemann surface ($m=1$), the above equation becomes Donaldson's generalization of the Wess--Zumino--Witten equation \cite{Do99,BK12}.
Furthermore, when $D \subset \Bbb C$ is an annulus, and $v$ is invariant under rotation of the annulus, the above equation reduces to the equation of geodesics in the space of K\"ahler metrics \cite{Ma87,Se92,Do99,Ch00}. 

When $v$ is smooth and non-degenerate in the $X$-fibers, it is known that the solution $u$ is $C^{1,\alpha}$ \cite{Bl13}. Though continuous boundary data has not been explicitly considered in the past, one can obtain the analogous result after an application of the maximum principle and approximation techniques of K\"ahler geometry:

\begin{theorem} \label{thm: deg_CMAE_cont_main}Given $v \in C(\partial{D} \times X)$ such that $v_z=v(z,\cdot) \in \textup{PSH}(X,\omega)$ for $z \in \partial D$, the Dirichlet problem \eqref{eq: DR_equation} has a unique solution $u  \in C(\overline{D} \times X) \cap \textup{PSH}(D \times X, \pi^* \omega)$. 
\end{theorem}

By estimates of Blocki \cite{Bl11}, in case $v$ is $C^{0,1}$ one can show that $u \in C^{0,1}$ as well. For brevity, we will focus on continuous solutions in this work, and we do not elaborate on regularity further.

As is well known, $\textup{PSH}(X,\omega) \cap C(X)$ is a subset of the Mabuchi metric completion of $\mathcal H_\omega$ \cite{Da17} (for a recent survey see \cite{Da19}). Consequently, one can think of  the solution $u$ to \eqref{eq: DR_equation} as a map $u_z:D \to  \overline{{\mathcal H}_\omega}$. As shown in Section 2, when $D$ is special, the map $z \to u_z$ is closely related to flat embeddings of convex domains into (the completion of) $\mathcal H_\omega$, hence  its understanding is geometrically well motivated.

\vspace{-0.1in}\paragraph{The dual Fubini--Study/Hilbert maps and Finsler quantization.} K\"ahler quantization has a long history going back to predictions of Yau \cite{Ya87} and early work of Tian \cite{Ti88, Ti90}, with refinements by Catlin \cite{Ca97}, Zelditch \cite{Ze98}, Lu \cite{Lu00} and many others. There is a canonical quantization scheme of  $\mathcal H_\omega$, whereby the infinite dimensional space $\mathcal H_\omega$ is approximated by the finite dimensional spaces of Hermitian metrics $\textup{Herm}(H^0(X,L^k))$ or $\textup{Herm}(H^0(X,L^k \otimes K_X))$. This scheme has been recently extended to the metric completion \cite{DLR18}, following up on work of many authors including \cite{Do01}, \cite{PS06}, \cite{SZ10}, \cite{Brn09b}, \cite{CS12}. 

In case $D$ is a Riemann surface, the quantization scheme of $z \to u_z$ was explained in embryonic form in \cite[Section 2.2.2]{BK12}. There the authors used the spaces of Hermitian metrics $\textup{Herm}(H^0(X,L^k \otimes K_X))$ to quantize, and they also speculated on the possibility of quantization in case $D$ is higher dimensional. Below we show that this can be carried out, however one needs to quantize using the Finsler(!) metrics of the dual space $H^0(X,L^k)^*$ instead, without any twisting by $K_X$. Indeed, it seems that there are not enough Hermitian metrics to quantize in case $\dim D > 1$. 

As a novelty of our work, we now point out that  one can find a natural extension of the classical Hilbert and Fubini--Study maps to the dual Finsler setting. We denote by 
\begin{align*}
&(1)\text{ $\mathcal{H}_k$ the space of Hermitian metrics on $H^0(X,L^k)$.}\\    
&(2)\text{ $\mathcal{N}_k$ the space of Finsler norms on $H^0(X,L^k)$.}\\    
&(3)\text{ $\mathcal{M}_k$ the space of psh Fisnler metrics on $H^0(X,L^k)$.}    
\end{align*}
The third one simply consists of Finsler metrics that are also psh on $H^0(X,L^k)$. Since norms are convex, and convex functions are psh, we have obvious inclusions $\mathcal{H}_k \subset \mathcal N_k \subset \mathcal M_k$.

Similarly we denote by $\mathcal{H}^*_k/\mathcal N^*_k/\mathcal M^*_k$ the space of Hermitian metrics/Finsler norms/psh Finsler metrics on the dual vector space $H^0(X,L^k)^*$. Again, we have the inclusions $\mathcal{H}_k^* \subset \mathcal N_k^* \subset \mathcal M_k^*$.

The classical Hilbert map $H_k: \mathcal{H}_{\omega}\rightarrow \mathcal{H}_k$ is given by the formula 
$$H_k(u)(s,s)=\int_X h^k(s,s)e^{-ku}\omega^n, \text{ for } s\in H^0(X,L^k).$$ 
The dual Hilbert map is $H_k^*:\mathcal{H}_{\omega}\rightarrow \mathcal{H}^*_k$, defined via dualization: 
\begin{equation}\label{eq: Hilb_k_usual_dual}
H^*_k(u) = H_k(u)^*.
\end{equation}
In the opposite direction, we have the well known Fubini--Study map $FS_k: \mathcal{H}_k \rightarrow \mathcal H_\omega$:
\begin{equation}\label{eq: FS_k_usual}
FS_k(G)=\frac{1}{k}\log \sup_{s\in H^0(X,L^k),G(s)\leq 1}h^k(s,s).
\end{equation}
One can similarly define the dual Fubini--Study map 
$FS^*_k: \mathcal{H}^*_k \rightarrow \mathcal H_\omega$, given by the formula 
$FS^*_k(G) := FS_k(G^*).$ 

We note that the operator $H_k$ is monotone decreasing in the sense that if $u_1\leq u_2$ then $H_k(u_1)(s,s)\geq H_k(u_2)(s,s)$ for any $s\in H^0(X,L^k)$. Also, $FS_k$ is monotone decreasing in the sense that if $G_1\leq G_2$ then $FS_k(G_1)\geq FS_k(G_2)$. On the other hand, the dual maps $H^*_k$ and $FS^*_k$ are monotone increasing due to the fact if a norm is large then its dual norm will be small. 

When trying to extend $FS^*_k$ to $\mathcal{N}^*_k$ and more generally, $\mathcal{M}^*_k$, a different point of view is necessary, as we now elaborate. Starting with an arbitrary global non-vanishing section of $(L^k)^*$ and dividing it by its length using the metric $(h^*)^k$, one obtains a global discontinuous section $s_k^*:X \to (L^k)^*$ such that $(h^*)^k(s_k^*(x),s_k^*(x))=1$ for all $x \in X$. Using this section one can introduce the evaluation map $\hat s_k^*: X \to H^0(X,L^k)^*$ given by the formula $\hat{s}_k^*(x)(\sigma) := s_k^*(\sigma(x))$, $\sigma \in H^0(X,L^k)$. 

We define $FS_k^*:\mathcal M_k^* \to \textup{PSH}(X,\omega)$ by the following formula
\begin{equation}\label{eq: FS_k_def_general}
FS^*_k(G^*)(x)= \frac{2}{k} \log\big[G^* (\hat{s}^*_k(x))\big], \ x \in X.
\end{equation}
It is not hard to see that  this definition is independent of the choice of $s_k^*$. We will show that it is consistent with the definition of \eqref{eq: FS_k_usual} and  $FS^*_k(G^*) \in \textup{PSH}(X,\omega)$ (Lemmas \ref{lem: KRW-formula} and \ref{lem: FS_k_PSH}). From the defining formula (\ref{eq: FS_k_def_general}), this extended Fubini--Study map $FS^*_k$ is also monotone increasing.

Let $v \in C(\partial{D} \times X)$ such that $v_z:=v(z,\cdot) \in \textup{PSH}(X,\omega), \ z \in \partial D$. We consider the trivial bundle $D\times H^0(X,L^k)^*$, and the following families of Griffiths negative Finsler metrics/norms on this bundle:
$$F_v^{\mathcal N,k}:= \{D \ni z \to U_z \in \mathcal N_k^* \textup{ is Griffiths negative and } \limsup_{z \to \partial D} U_z \leq H_k^*(v)\}.$$
$$F_v^{\mathcal M,k}:= \{D \ni z \to U_z  \in \mathcal M_k^*\textup{ is Griffiths negative and } \limsup_{z \to \partial D} U_z \leq H_k^*(v)\}.$$

As pointed out after \eqref{eq: Griff_Extr_Metric_Perron} and \eqref{eq: Griff_Extr_Norm_Perron}, these families are  stable under taking the supremum, allowing us to consider their upper envelopes: 

\begin{equation}\label{eq: Griff_extr_metric_norm_Perron_def_intr}
U^{\mathcal M,k}:= \sup_{V \in F^{\mathcal M,k}_v} V, \ \ \ \ \ U^{\mathcal N,k}:= \sup_{V \in F^{\mathcal N,k}_v} V, \ \ z \in D. 
\end{equation}

By the comments following Theorem \ref{thm: main PDE theorem}, $U^{\mathcal M,k}$ is the Griffiths extremal Finsler metric assuming the boundary values $H_k^*(v)$.  $U^{\mathcal N,k}$ also assumes the correct boundary values (Proposition \ref{prop: bdry data}), and trivially $U^{\mathcal N,k} \leq U^{\mathcal M,k}$, but (as discussed before) in general $U^{\mathcal N,k}$ and $U^{\mathcal M,k}$ are different. Despite this, in our main result we show that both of these envelopes tend to the solution $u$ of \eqref{eq: DR_equation} in the large $k$-limit:

\begin{theorem}\label{thm: C_0_quantization_main} Given $v \in C(\partial{D} \times X)$ such that $v_z \in \textup{PSH}(X,\omega), \ z \in \partial D$,  we have  that:\\
\noindent (i) $\|FS_k^*(U^{\mathcal N,k}) - u \|_{C^0(D \times X)} \to 0$ as $k \to \infty$,\\
\noindent (ii) $\|FS_k^*(U^{\mathcal M,k}) - u \|_{C^0(D \times X)} \to 0$ as $k \to \infty$,\\
where $u$ is the solution to \eqref{eq: DR_equation} and $U^{\mathcal N,k}/U^{\mathcal M,k}$ are the envelopes of Griffiths negative norms/metrics from \eqref{eq: Griff_extr_metric_norm_Perron_def_intr}.
\end{theorem}

Although $U^{\mathcal N,k}$ and $U^{\mathcal M,k}$ are different in general, they do agree when $\dim D=1$. In fact, in this particular case both $U^{\mathcal N,k}$ and $U^{\mathcal M,k}$ are Hermitian metrics.
Indeed, this is a consequence of the Wiener--Masani type decomposition of the boundary data $H_k^*(v)$, as elaborated in \cite{CS93,Do92}. Due to this observation,  Theorem \ref{thm: C_0_quantization_main} recovers  well known results on quantization of Mabuchi geodesics and solutions to Wess--Zumino--Witten equations \cite{Brn09b,BK12}. 

For smooth $v$ it is expected that the $C^0$-convergence of the above theorem can be upgraded to $C^{1,\alpha}$-convergence, at least when $\partial D$ is Levi flat. However this is an open question even in the case when $D$ is a Riemann surface, and will likely require a refined analysis of Bergman kernels that is beyond the scope of the current work. We hope to return to this sometime in the future.

\vspace{-0.1in}\paragraph{Organization.} To provide motivation, in Section \ref{sec motivation} we point out the connection between Mabuchi flatness and the complex Monge--Amp\`ere equations considered in this work. Theorem \ref{thm: main PDE theorem} is proved in Section \ref{sect Griff neg} (Theorem \ref{thm: main_PDE_noncompact}). Theorem \ref{thm: deg_CMAE_cont_main} is proved in Section \ref{subsec degenerate CMAE} (Theorem \ref{thm: deg_CMAE_cont_boundary_data}). Theorem \ref{thm: C_0_quantization_main} is proved in Section \ref{sec Griff neg and quantization} (Theorem \ref{thm: C_0_quantization}). 

\vspace{-0.1in}\paragraph{Acknowledgments.}We would like to thank L\'aszl\'o Lempert for numerous suggestions improving the presentation of the paper. We thank the referees for careful reading and helpful comments. The  first  named  author  has  been  partially  supported  by  NSF  grants  DMS-1610202 and DMS-1846942(CAREER). The second named author has been partially supported by NSF grant DMS-1764167.

\section{Motivation: complex Monge--Amp\`ere equations and Mabuchi flatness}\label{sec motivation}

In this short section, we equip the space of K\"ahler potentials $\mathcal H_\omega$ with the usual Mabuchi $L^2$ geometry \cite{Ma87,Se92,Do99}. Given a compact convex set $\Omega \subset \Bbb R^m$ with non-empty interior, we would like to find a flat embedding $\Omega  \ni x \to u_x \in \mathcal H_\omega$. Here flatness is understood in the sense of metric spaces: the image of any segment in $\Omega$ under $x \to u_x$ is a geodesic of $\mathcal H_\omega$ \cite[Chapter II.2]{BH99}. For related work on quantizing harmonic maps into $\mathcal H_\omega$ we refer to \cite{RZ10}.

Given that $(\mathcal H_\omega,\langle\cdot,\cdot\rangle)$ is non-positively curved, the study of flat maps  plays a special role in the exploration of $\mathcal H_\omega$ and its Mabuchi completion $(\mathcal E^2,d_2)$, which is a CAT(0) space \cite{Da17}.  

There is a natural connection with complex Monge--Amp\`ere equations, specifically \eqref{eq: DR_equation}. As our goal is to provide motivation, we assume for simplicity that $\Omega \ni x \to u_x \in \mathcal H_\omega$ is a smooth flat embedding. 
Introducing $\Omega^\Bbb C := \Omega + i \Bbb R^m \subset \Bbb C^m$ (the Bochner tube with base $\Omega$), and the projection $\pi: \Omega^\Bbb C \times X \to X$, one can consider the ``complexification" $\Omega^\Bbb C \ni z \to u_z \in \mathcal H_\omega$, where $u_z := u_{\textup{Re }z}$. For sake of well-posedness we assume the positivity condition $\pi^* \omega + i\partial_{\Omega^\Bbb C \times X} \bar \partial_{\Omega^\Bbb C \times X} u \geq 0$ on $\Omega^\Bbb C \times X$, where $u(z,x)=u_z(x)$. 

Due to flatness, for any $a,b \in \Omega$, the curve $t \to u_{a + t(b-a)}$ is a Mabuchi geodesic. Due to positivity, we obtain that the (1,1)-form $\pi^*\omega + i\partial_{\Omega^\Bbb C \times X} \bar \partial_{\Omega^\Bbb C \times X} u$ has a zero eigenvalue for all $(z,x) \in \Omega^\Bbb C \times X$.  Consequently, we have that 
\begin{equation}\label{eq: CMAE_motivation}
(\pi^* \omega + i\partial_{\Omega^\Bbb C \times X} \bar \partial_{\Omega^\Bbb C \times X}  u)^{m+n} = 0 \ \ \textup{ on } \ \ \Omega^\Bbb C \times X.
\end{equation}
Clearly, equation \eqref{eq: CMAE_motivation} does not characterize the flatness condition for a smooth map $x \to u_x$. However flatness most often leads to over-determined problems in geometric analysis. On the other hand, the weaker condition of the above equation does allow for a robust setup, as explored in the previous sections. This same exact PDE was considered in \cite{DR16}, and one can think of the Dirichlet problem \eqref{eq: DR_equation} as trying to find (weak) flat maps into the space of K\"ahler metrics with prescribed boundary data.

Unfortunately $\mathcal H_\omega$ lacks smooth geodesics, so one is ultimately interested in flat embeddings into the metric completion $(\mathcal E^2,d_2)$, that is a geodesic CAT(0) metric space \cite{Da17}.  This motivates our consideration of Bedford--Taylor solutions to \eqref{eq: CMAE_motivation} throughout this paper.

It remains an interesting question to study the additional constraints under which solutions to \eqref{eq: CMAE_motivation} are always flat. One such condition is asking for affinity of  $x \to I(u_x), \ x \in \Omega,$ where $I$ is the Monge--Amp\`ere energy, and we hope to return to this problem in a future publication.

\section{Preliminaries}

In this section we collect known facts from the literature, introduce and explain our own terminology, and prove some preliminary technical results. 

Let $X$ be a complex manifold, we say that $u: X \to [-\infty,\infty)$ is  \emph{quasi-plurisubharmonic} (qpsh) if there exists an atlas $\{U_j\}_j$ for $X$, and $f_j \in C^\infty(U_j)$ such that $f_j + u|_{U_j}$ are psh functions on each $U_j$. If each $f_j$ can be taken to be zero, then $u$ is psh on $X$.

If $\omega$ is a K\"ahler form on $X$, then we say that $u: X \to [-\infty,\infty)$ is \emph{$\omega$-psh} $(u \in \textup{PSH}(X,\omega))$ if there exists an atlas $\{U_j\}_j$ for $X$, and $v_j \in C^\infty(U_j)$ local potentials for $\omega$ (i.e., $\omega|_{U_j} = i\partial \bar \partial v_j$), such that $v_j + u|_{U_j}$ are psh functions on each $U_j$.

\subsection{Approximation of families of quasi-psh potentials}

In this subsection we prove an elementary approximation result for a family of qpsh functions, that will be used multiple times in this work. Let $\pi: X \to B$ be a smooth submersion with $X,B$ compact smooth manifolds and each fiber $X_b := \pi^{-1}(b)$  a complex manifold.

We assume that $B$ has a fixed Riemannian metric $r$.
The compact fibers $X_b$ are diffeomorphic to a fixed compact complex manifold $Z$ (when $B$ is connected). However typically they are not biholomorphic to $Z$, in fact the complex structure may not even change continuously, the way we set things up. For this reason, we introduce below the ad-hoc notion of split holomorphicity, that allows for smooth families of biholomorphisms between nearby fibers, and will suffice for our purposes.

Let $TB \times_B X \to B$ be the fibered direct product of the tangent bundle $TB \to B$ and $X \to B$.
We say that $\pi$ \emph{splits holomorphically} if the zero section of the tangent bundle $TB$ has an open neighborhood $U$ for which there exists a smooth map $F: U \times_B X \to X$ with the \vspace{0.1cm} following properties (we denote $U_b:=U\cap T_bB $ for $b\in B$): \medskip

\noindent (i)\vspace{0.15cm}  $exp_b : U_b \to exp_b(U_b)$ is a diffeo, $b \in B$, where $exp_b$ is the Riemannian exponential map of $r$; \\ 
\noindent (ii) if\vspace{0.15cm} $v \in U_b$, $b \in B$ then $F(v, \cdot): X_b \to X_{exp_b(v)}$ is a biholomorphism; \\ 
\noindent (iii) for $0 \in U_b \subset T_b B$, $F(0,\cdot): X_b \to X_b$ is the identity map. \medskip

\begin{remark}\label{rem: split_holom_ex} We give two examples of submersions that split holomorphically:

\noindent (a) if $X = B \times Z$ then clearly $\pi:= pr_1 :B \times Z \to B$ splits holomorphically, with $U$ being the neighborhood of the zero section in $TB$, where $exp_b: U_b \to exp_p(U_b)$ is a diffeomorphism, and $F(v,b,z)=(exp_b(v),z)$ for any $v \in U_b$ and $z\in Z$. 

\noindent (b) a more intricate example is when $E \to B$ is a smooth $\Bbb C$-vector bundle with rank $r$. Then $\pi:\Bbb P(E) \to B$ is again seen to split holomorphically, where $\mathbb{P}(E)$ is the projectivization of $E$. Indeed, let $U$ be the neighborhood of the zero section in $TB$, where $exp_b: U_b \to exp_p(U_b)$ is a diffeomorphism.

Now let $g$ be a hermitian metric on $E$. Then $g$-parallel transport along geodesics of $r$ gives a $\Bbb C$-linear isomorphism between $T_{v}: E_b \to E_{exp_b(v)}$ for any $v \in U_b, \ b \in B$. After projectivization, we get a biholomorphism $\Bbb P T_{v}: \Bbb P(E_b) \to \Bbb P(E_{exp_b(v)})$, hence $F(v,x):= \Bbb P T_{v}(x), \ (v,x) \in U \times_B \mathbb{P}(E)$ satisfies the requirements in the definition of holomorphic splitness.
\end{remark}

\begin{theorem}\label{thm: approx_family} Let $\pi: X \to B$ be a smooth submersion with $X,B$ compact smooth manifolds and each fiber $X_b := \pi^{-1}(b)$  a complex manifold. Assume that $\pi$ splits holomorphically.
Let $\alpha$ be a smooth real $2$-form on $X$ such that $\alpha_b := \alpha|_{X_b}$ is K\"ahler on each $X_b$.

Let $u \in C(X)$ such that $u_b:= u|_{X_b} \in \textup{PSH}(X_b,\alpha_b)$. Then there exists $u^k \in C^\infty(X)$ such that $u^k \searrow u$ and $\alpha_b + i\ddbar u^k_b$ is K\"ahler on $X_b$ where $u^k_b:=u^k|_{X_b}$.
\end{theorem}

\begin{proof}Let us fix $b \in B$ momentarily, and let $k \in \Bbb N$. By \cite[Theorem 1]{BK07} there exists $\tilde u^k_b \in C^\infty(X_b)$ such that $\alpha_b + i\ddbar \tilde u^k_b$ is K\"ahler on $X_b$ and $\sup_{X_b}|\tilde u^k_b - u_b|<1/2k$.

By split holomorphicity, and since $\alpha$ is smooth on $X$ we obtain that $\alpha_{exp_b(v)} + i\ddbar \tilde u^k_b(F(v,\cdot)^{-1})>0$ for $v$ in a neighborhood of $0 \in T_b B$. Since $exp_b$ is a diffeomorphism near $0 \in T_b B$, we can find a neighborhood $V_b \subset B$ of $b$, such that $\tilde u^k_b$ extends to  a smooth function $\check u^k: \pi^{-1}(V_b) \to \Bbb R$ such that $\alpha_{b'} + i\ddbar \check u^k_{b'}$ is a K\"ahler form on $X_{b'}$ for all $b' \in V_b$. Moreover, since  $u \in C(X)$, we can further assume that $\sup_{X_{b'}}|\check u^k_{b'} - u_{b'}|<1/k$ for $b'\in V_b$. This gives the local version of Theorem \ref{thm: approx_family} on $\pi^{-1}(V_b)$.

Now we globalize the above local construction. Let $V_{b_1},\ldots, V_{b_m}$ be a finite cover of $B$, and let $\check u^k_j$ be the corresponding approximants on each $\pi^{-1}(V_{b_j})$ constructed above. Let $\rho_1,\ldots, \rho_m$ represent a partition of unity for $B$ that is subordinate to the cover  $V_{b_1},\ldots, V_{b_m}$.

We see that $\check u^k(x) := \sum_j \rho_j( \pi(x)) \check u^k_j(x) \in C^\infty(X)$, $\check u^k$ converges uniformly to $u$ on $X$, moreover  $\alpha_{b} + i\ddbar \check u^k_{b}$ is a K\"ahler form on $X_{b}$ for all $b \in B$.
Due to uniform convergence, there exist constants $c_k \searrow 0$ such that $u^k:= \check u^{k} + c_k \searrow u$, finishing the argument. 
\end{proof}

\subsection{Griffiths negativity of Finsler bundles} \label{subsec griff neg of finsler}

Let $Y$ be an $m$-dimensional complex manifold. We start with a discussion on the connection between a holomorphic vector bundle $E \to Y$ of rank $r$ and its tautological bundle $L(E) \to \Bbb P(E)$ in the Finsler context. Here $\Bbb P(E)$ is the projectivization of $E$ and $L(E) \to \Bbb P(E)$ is the tautological line bundle. There is a natural map $\Tilde{p}: L(E)\to E$ mapping $([x],\lambda x)$ to $\lambda x$, which is biholomorphic away from the zero sections. As observed by Kobayashi, Hermitian metrics on $L(E)$ are in one-to-one correspondence with Finsler metrics on $E$! 

In \cite{Ko75, Ko96}, Kobayashi defined for a strongly pseudoconvex smooth Finsler metric $f$ a notion of curvature, and from his computation one can deduce that $f$ is negatively curved in the sense of Kobayashi if and only if the associated metric $f_L=f\circ \tilde{p}$ on $L(E)$ has negative curvature, i.e., in local coordinates $\log (f_L)$ is psh.

For us, a Finsler metric $f$ is Griffiths negative if it is psh on the total space $E$ of the bundle. Notice that we do not assume strong pseudoconvexity or smoothness. By the proposition below, our definition indeed extends Kobayashi's.

\begin{prop}\label{prop: Griff_Kob_negative_equiv} Let $f$ be a Finsler metric on $\pi: E \to Y$. Then the following are equivalent:\\
\noindent (i) $f$ is Griffiths negative.\\
\noindent (ii) $f$ is plurisubharmonic on the total space of $E$.\\
\noindent (iii) $\log f$ is plurisubharmonic on the the total space of $E$. \\
\noindent (iv) $f_L$ has negative curvature on the line bundle $L(E)$.\\
\noindent (v) $f_L$ is plurisubharmonic on the total space of $L(E)$.\\
\noindent (vi) $\log f_L$ is plurisubharmonic on the total space of $L(E)$. 
\end{prop}

\begin{proof} The equivalence between (i) and (ii) follows from our definition of Griffiths negativity. Now we show that (ii) implies (iii). Let $G: V \to E$ be holomorphic, with $V \subset \Bbb C$ an arbitrary open set. We need to show that $\log f(\pi(G),G)$ is psh on $V$. By \cite[Theorem J.7]{Gu90} it is enough to show that for any open set $W\subset V $ and any $g: W \to \Bbb C$ holomorphic, the function $\log f(\pi(G),G) - \textup{Re } g$ satisfies the maximum principle on $W$. By homogeneity of  $f$ we have that
$$\log f(\pi(G(z)),G(z)) - \textup{Re }g(z)=\log f(\pi(G(z)),G(z)e^{-g(z)}).$$
By assumption $z \to f(\pi(G(z)),G(z)e^{-g(z)})$ is psh on $W$, hence it satisfies the maximum principle on $W$. Consequently we then obtain that the logarithm of this expression satisfies the maximum principle as well, implying (iii), as desired.

That (iii) implies (ii) follows from the fact that the composition of psh functions with increasing convex functions stays psh.

That (ii) is equivalent with (v)  (and (iii) is equivalent with (vi)) follows from the fact that $\tilde p: L(E) \to E$ is a biholomorphism away from the zero sections of these bundles, which themselves are pluripolar sets.

Now we argue that (iii) implies (iv). Let $w \in \Bbb P(E)$ and  $U \subset \Bbb P(E)$ an open neighborhood of $w$, where $L(E)$ has a nonvanishing section $s: U \to L(E)$. Now (iv) follows as $z \to  \log f_L(s(z))$ is psh on $U$, since (iii) implies that so is $z \to \log f(\pi(\tilde{p}(s(z))),\tilde{p}(s(z)))$, and $f_L(s)=f(\pi(\tilde{p}(s)),\tilde{p}(s))$.

To finish the proof, we argue that (iv) implies (vi). Let $U \times \Bbb C^r$ be a local trivialization of $E|_U$, where we assume that $U \subset Y$ is a coordinate patch. 

Let $H_j \subset \Bbb C^{r}$ be the set of vectors whose $j$-th coordinate is equal to $1$, providing the classical coordinate coverings of $\Bbb C \Bbb P^{r-1}$.
Then (iv) implies that
$U \times H_j \ni (z,\xi) \to \log  f_L(z,[\xi])(\xi)$ is psh. Consequently, $U \times H_j \times \Bbb C \ni (z,\xi,s) \to \log  f_L(z,[\xi])(s\xi) = \log|s| + \log f_L(z,[\xi])(\xi)$ is psh as well. Since the sets of the type $U \times H_j \times \Bbb C$ provide coordinate charts near all points of $L(E)$, the proof of (vi) is finished.
\end{proof}

\section{The Dirichlet problem for Griffiths extremality.} 
\label{sect Griff neg}

Now we consider $D \subset Y$ a relatively compact strongly pseudoconvex smooth domain. This simply means that there exist an open set $D'\subset Y$ such that $D' \supset \overline{D}$ and $\rho \in C^\infty(D')$ such that $i\ddbar \rho >0$, $\rho^{-1}(-\infty,0) = D$, $\rho^{-1}(0)=\partial D$, and $d\rho\neq 0 $ on $\partial D$. We fix such a $\rho$ for this whole paragraph. In our first lemma we point out that we can pick a smooth Hermitian metric on $E$ that is Griffiths strictly negative in a neighborhood of $D$:

\begin{lemma}\label{lem: background_form_construct} There exists an open neighborhood $\tilde D \supset D$ and a smooth Hermitian metric $h$ on $E|_{\tilde D}$ that is Griffiths strictly negative. 
\end{lemma}

\begin{proof} Pick an arbitrary smooth Hermitian metric $\tilde h$ on $E \to Y$. For $k \in \Bbb N$ high enough $h = \tilde h e^{k\rho }$ satisfies the requirement of the lemma, because $i\Theta(h)=i\Theta(\tilde{h})-ki\ddbar \rho\otimes Id_E$ (see  \cite[Chapter V]{De12}).
\end{proof}

The above result allows to shrink $Y$ (without changing $D$), so that $(E,h)$ is Griffiths strictly negative and smooth globally. By picking $Y := \rho^{-1}(-\infty,\varepsilon)$, we can assume that $Y$ is Stein. We make these assumptions throughout this section.  This will not lead to loss of generality, as our focus is on the restricted bundle $E|_{\overline{D}}$.  
Moreover, let $h_L=h\circ \tilde{p}$ be the associated metric on $L(E)$  and let $\alpha = - \Theta(h_L)$ denote the negative of the Chern curvature (1,1)-form of $h_L$. By Lemma \ref{lem: background_form_construct}, $\alpha$ is a K\"ahler form on $\Bbb P(E)$.

We will denote by $F^\mathcal M_-$ the collection of Griffiths negative Finsler metrics on $E \to D$. 
Given a Finsler metric $g$ on the boundary $E|_{\partial D}$, we are interested in finding a Griffiths negative Finsler metric $f$ on $E|_{\overline{D}}$ assuming the values of $g$ on $E|_{\partial D}$ that is extremal, as defined in the introduction. For this it is necessary to impose the condition $g_z=g(z,\cdot) \in \textup{PSH}(E_z), \ z \in \partial D$.

By $F^\mathcal M_g$  we denote the Finsler metrics $v \in F^\mathcal M_-$ such that  $v \leq g$ on $E|_{\partial D}$. By this last condition we mean that $\limsup_{E|_D \ni (y,\xi) \to (y',\xi')} v(y,\xi) \leq  g(y',\xi')$ for any $(y',\xi') \in E|_{\partial D}$. 
As $F^\mathcal M_g$ is stable under maximum it makes sense to consider a Perron type envelope $f_g$ associated to $g$:
\begin{equation}\label{eq: Perron_Finsler_def}
f_g := \sup_{v \in F^\mathcal M_g} v.
\end{equation}

There are a number of things we would like to know about $f_g$: does it assume the right boundary values? Is the supremum finite? More importantly, is $f_g$ an element of $F^\mathcal M_g$? Does it uniquely solve some PDE?

As we will show below, the answer to all these questions is in the affirmative.  
Before we introduce the Dirichlet problem(s) that our Griffiths extremal metric $f_g$ will solve, some preliminary work is necessary. For any $f \in F^\mathcal M_-$ we have that 
$$f_L = h_L e^{\varphi_f},$$ 
for some  function $\varphi_f: \Bbb P(E)|_D \to \Bbb R$. We obtain that $\alpha + i\ddbar \varphi_f = -\Theta(f_L) \geq 0$ on $\Bbb P(E)|_D$.
 
In particular, since there is a one-to-one correspondence between Finsler metrics $f$ on $E\to D$ and Hermitian metrics $f_L$ on $L(E) \to \Bbb P(E)|_D$, we get that there is a one-to-one correspondence between $\gamma \in F^\mathcal M_-$ and $\varphi_\gamma \in \textup{PSH}(\Bbb P(E)|_D,\alpha)$ (Proposition \ref{prop: Griff_Kob_negative_equiv}). We can take this correspondence one step further:

\begin{lemma}\label{lem: bijection_metric_psh}There is a one-to-one correspondence between the metrics $\eta \in F^\mathcal M_g$ and the potentials $\varphi_\eta \in \textup{PSH}_{\varphi_g}(\Bbb P(E)|_D,\alpha)= \{\chi \in \textup{PSH}(\Bbb P(E)|_D,\alpha) \textup{ s.t. } \limsup_{\Bbb P(E)|_{D} \ni z \to y} \chi(z) \leq \varphi_g:=\log \big(\frac{g_L}{h_L}\big)(y), \ y \in \Bbb P(E)|_{\partial D}\}$.
\end{lemma}
\begin{proof} By the above discussion we only need to argue that  $\limsup_{E|_D \ni (y,\xi) \to (y',\xi')} \eta(y,\xi) \leq  g(y',\xi')$ is equivalent with $\limsup_{\Bbb P(E)|_{D} \ni a \to b} \varphi_\eta(a) \leq \log \big(\frac{g_L}{h_L}\big)(b)$  for any $b \in \Bbb P(E)|_{\partial D}$ and $\eta \in F^\mathcal M_-$.

The forward direction is elementary, so we will only argue the backward direction.
Let us assume that $\varphi_\eta \in \textup{PSH}_{\varphi_g}(\Bbb P(E)|_D,\alpha)$.

By going back and forth between $E \to D$ and $L(E) \to \Bbb P(E)|_D$ one can see that 
if $\xi' \in E_{y'}$ with $y' \in \partial D$ and $\xi'\neq 0$, then 

$$\limsup_{E|_D \ni (y,\xi) \to (y',\xi')} \eta(y,\xi)=\limsup_{E|_D \ni (y,\xi) \to (y',\xi')} \eta_L(y,[\xi],\xi)=\limsup_{E|_D \ni (y,\xi) \to (y',\xi')} h_L e^{\varphi_\eta}(y,[\xi])(\xi) \leq  g(y',\xi').$$

What remains to argue is the case $\xi' =0$. We argue this part by contradiction: let $y'_j \to y'$ and $\xi'_j \to \xi'= 0$ such that $\limsup \eta(y_j',\xi'_j)> \varepsilon$. In fact, after taking a subsequence, we can even assume that $\lim \eta(y_j',\xi'_j)> \varepsilon$, and each $\xi'_j \in E_{y'_j}$ is non-zero. Homogeneity allows to write:
\begin{equation} \label{eq: zeta_def}
\eta(y'_j,\xi'_j) = h(y'_j,\xi'_j) \eta(y'_j,\zeta'_j), 
 \ \ \textup{ where } \ \zeta'_j = \frac{\xi'_j}{h(y'_j,\xi'_j)}.
 \end{equation}
Since the $h$-length of $\zeta'_j$ is equal to $1$, after possibly taking subsequence, there exists $\zeta' \in E_{y'}$ (also with $h$-length equal to $1$) such that $\zeta'_j \to \zeta'$. Then the first step gives $\limsup_j \eta(y'_j,\zeta'_j)   \leq  g(y',\zeta')$.

Using \eqref{eq: zeta_def} we arrive at $\lim_j \eta(y'_j,\xi'_j)=\lim_j h(y'_j,\xi'_j)\eta(y'_j,\zeta'_j) \leq    g(y',\zeta') \limsup_j h(y'_j,\xi'_j) =0,$  a contradiction.
\end{proof}

Before looking at \eqref{eq: Griff_extr_equation_intr}, we need to consider the analgous Dirichlet problem on the projectivization, with solutions interpreted in the language of Bedford--Taylor theory: for a function $\psi$ on $\Bbb P(E)|_{\overline{D}}$

\begin{equation}\label{eq: Giff_extremality_system}
    \begin{cases}
    (\alpha + i\ddbar \psi)^{m+r-1}=0 \text{  on $\Bbb P (E)|_D$},\\
    {\psi} \in \textup{PSH}(\Bbb P (E)|_D,\alpha) \cap L^\infty,\\
     \psi = \log \frac{g_L}{h_L} \text{  on $\Bbb P(E)|_{\partial D}$}.
        \end{cases} 
\end{equation}

\begin{theorem}\label{thm: extremal_PDE} $f_g$ is the unique element of $F^\mathcal M_g$ for which $\varphi_{f_g}$ is bounded and solves \eqref{eq: Giff_extremality_system}. In addition, $\varphi_{f_g}$ and $f_g$ are continuous up to the boundary.
\end{theorem}

\begin{proof} Uniqueness of solutions  to \eqref{eq: Giff_extremality_system} is a consequence of \cite[Theorem 21]{Bl13}. Although this uniqueness theorem is stated for continuous solutions, its proof goes through for bounded solutions as long as $\log \frac{g_L}{h_L}$ is continuous on $\Bbb P(E)|_{\partial D}$.

To address existence, we first construct a sequence of approximate $C^{0,1}$ subsolutions. 

First we construct a sequence of smooth approximate boundary data. This is a standard technical argument, so we will be brief. Since $\varphi_g= \log \frac{g_L}{h_L} \in C(\Bbb P(E)|_{\partial D})$ and $\varphi_g(z) \in \textup{PSH}(\Bbb P(E_z),\alpha_z)$ for all $z \in \partial D$, we can use Remark \ref{rem: split_holom_ex}(b) and Theorem \ref{thm: approx_family} to get $\chi^k  \in C^\infty(\Bbb P(E)|_{\partial D})$ such that $\alpha_z + i\ddbar \chi^k_z>0, \ z \in \partial D$ and $\chi^k \searrow \varphi_g$ uniformly. 

Now one extends $\chi^k$ smoothly to $\Bbb P(E)|_{\overline{D}}$, such that  $\alpha_z + i\ddbar \chi^k_z>0, \ z \in \overline{D}$. Since we are only asking for positivity in the fiber directions, this can be done by extending $\chi^k$ first arbitrarily, then multiplying this extension by an appropriate smooth cutoff function  of the boundary $\Bbb P(E)|_{\partial D}$ that is constant on the fibers.

After adding $l\rho$ to this smooth extension (with $l > 0$ sufficiently big) we get that $\alpha + i\ddbar \chi^k >0$ on $\Bbb P(E)|_{\overline{D}}$. Since $\rho = 0$ on $\partial D$, this last step did not change the values of $\chi^k$ on $\partial D$.

Now let $\psi^k$ be the solution to the following PDE:
\begin{equation}\label{eq: Giff_extremality_system_k}
    \begin{cases}
    (\alpha + i\ddbar \psi^k)^{m+r-1}=0 \text{  on $\Bbb P (E)|_D$},\\
    {\psi^k} \in \textup{PSH}(\Bbb P (E)|_D,\alpha) \cap L^\infty,\\
     \psi^k = \chi^k \text{  on $\Bbb P(E)|_{\partial D}$}.
        \end{cases} 
\end{equation}
By Lemma \ref{lem: Lipschitz_solution} below, $\psi^k$ exists and is Lipschitz up to the boundary. Moreover, 
\begin{equation}\label{eq: Perron_approx}
\psi^k := \sup_{v \in \textup{PSH}_{\chi^k}(\Bbb P(E)|_{\overline{D}},\alpha)} v.
\end{equation}
Additionally, by the comparison principle \cite[Theorem 21]{Bl13} we have that the $\{\psi^k\}_k$ is a decreasing sequence of continuous functions, because so is the boundary data $\{\chi^k\}_k$. Moreover, since $\sup_{\Bbb P(E)|_{\partial D}}|\chi^k - \chi^l| \to 0$ as $k,l \to 0$ the same comparison principle implies that $\sup_{\Bbb P(E)|_{\overline{D}}}|\psi^k - \psi^l| \to 0$ as $k,l \to 0$. As a result $\{\psi^k\}_k$ is a Cauchy sequence with respect to uniform convergence, converging to $\psi \in \textup{PSH}(\Bbb P (E)|_D,\alpha) \cap C(\Bbb P (E)|_{\overline{D}})$. Hence $\psi$ is continuous and is equal to $\varphi_g$ on the boundary. Basic theorems of Bedford--Taylor theory now imply that $(\alpha + i\ddbar \psi)^{m+r-1}=0$. Lastly, \eqref{eq: Perron_approx} gives
$$\psi := \sup_{v \in \textup{PSH}_{\varphi_g}(\Bbb P(E)|_{\overline{D}},\alpha)} v.$$
Using Lemma \ref{lem: bijection_metric_psh} we obtain that $\psi = \varphi_{f_g}$, finishing the proof.
\end{proof}

\begin{lemma}\label{lem: Lipschitz_solution} Equation \eqref{eq: Giff_extremality_system_k} has a solution $\psi^k \in C^{0,1}(\Bbb P(E)|_{\overline{D}})$, with $\|\psi^k\|_{C^{0,1}}$ being controlled by $\| \chi^k\|_{C^{0,1}}$.
\end{lemma}

\begin{proof} This follows from \cite[Theorem 26]{Bl13}. To start we set up the system
\begin{equation}\label{eq: Giff_extremality_system_k_eps}
    \begin{cases}
    (\alpha + i\ddbar \psi_\varepsilon^k)^{m+r-1}=\varepsilon \alpha^{m+r-1} \text{  on $\Bbb P (E)|_D$},\\
    {\psi_\varepsilon^k} \in \textup{PSH}(\Bbb P (E)|_D,\alpha) \cap L^\infty,\\
     \psi_\varepsilon^k = \chi^k \text{  on $\Bbb P(E)|_{\partial D}$}.
        \end{cases} 
\end{equation}
Since the datum of this elliptic PDE is smooth, \cite[Theorem 26]{Bl13} and \cite[Theorem 19]{Bl13} is applicable to give a smooth solution $\psi_\varepsilon^k$, with $\|\psi_\varepsilon^k\|_{C^{0,1}}$ being controlled by $\| \chi^k\|_{C^{0,1}}$, and is independent of $\varepsilon>0$. After taking a subsequence, the Arzel\`a--Ascoli theorem guarantees that $ \psi_\varepsilon^k \to \psi^k$ uniformly on $\Bbb P(E)|_{\overline{D}}$, as $\varepsilon \to 0$, and $\|\psi^k\|_{C^{0,1}}$ is controlled by $\| \chi^k\|_{C^{0,1}}$. That $\psi^k$ solves \eqref{eq: Giff_extremality_system_k} follows from the convergence of complex Monge--Amp\`ere measures along uniformly convergent sequences of bounded potentials \cite{BT76}.
\end{proof}

We now start focusing on the Dirichlet problem of Theorem \ref{thm: main PDE theorem}, and we will eventually show that $f_g$ is the unique continuous solution to this PDE, that we now recall:
\begin{equation}\label{eq: Griff_extr_equation}
    \begin{cases}
    (i\partial \bar{\partial} f)^{m+r}=0 \text{  on $E|_D$}\\
    f \in \textup{PSH}(E|_D),\\
    f=g, \text{  on $E|_{\partial D}$}.
        \end{cases}
\end{equation}
Again, $g \in C(E|_{\partial D})$ is a Finsler metric satisfying $g_z := g(z,\cdot) \in \textup{PSH}(E_z), \ z \in \partial D$. 

Such Dirichlet problems are often solvable using a  Perron process that we now consider. Let $\textup{PSH}_g(E|_D)$ be the set of psh functions $u$ on $E|_D$ such that $\limsup_{E|_D \ni (z',\xi') \to (z,\xi) \in E|_{\partial D}}u(z',\xi') \leq g(z,\xi)$, and consider the following upper envelope:
$$u_g := \textup{usc}\Big(\sup_{v \in \textup{PSH}_g(E|_D)} v\Big),$$

Here $\textup{usc}(\cdot)$ is the upper semicontinuous regularization. Compared to \eqref{eq: Perron_Finsler_def}, we note that the elements of $\textup{PSH}_g(E|_D)$ are not homogeneous in the fibers of $E$. Additionally, we do not even know if $u_g$ is bounded above.  If one can show that $u_g \in \textup{PSH}_g(E|_{D})$, then automatically $(i\partial \bar{\partial} u_g)^{m+r}=0$ by the classical balayage argument of Bedford--Taylor \cite{BT76}. We confirm all of this and more in the main theorem of this section:

\begin{theorem}\label{thm: main_PDE_noncompact} $u_g$ is locally bounded above and $u_g \in \textup{PSH}_g(E|_{D})$. Moreover, $u_g = f_g$, automatically implying that $u_g$ solves \eqref{eq: Griff_extr_equation}. Lastly, $u_g$ is also the unique solution in $F^\mathcal M_g$ to \eqref{eq: Griff_extr_equation}.
\end{theorem}

\begin{proof} First some preliminary analysis, to argue that  $u_g \in \textup{PSH}_g(E|_D):$ since $Y$ is Stein, by Cartan's Theorem A \cite[Theorem 7.28]{Ho90}, there exist sections $s_1,\ldots,s_k \in H^0(Y,E)$ that span $E$ on $\overline{D}$. This gives a surjective morphism of bundles $\phi: \Bbb C^k|_{\overline{D}}\to E|_{\overline D}$ given by the formula $\phi(\lambda_1,\ldots,\lambda_k) = \sum_j \lambda_j s_j$. Moreover, we have that $v \circ \phi \in \textup{PSH}_{g \circ \phi}(\Bbb C^k|_{D})$ for any $v \in \textup{PSH}_g(E|_D)$.

Let $h_\lambda \in C(\overline{D}) \cap C^\infty(D)$ denote the harmonic function for which $h_\lambda|_{\partial D} = g\circ \phi(\cdot,\lambda)$. By the maximum principle for harmonic functions, for $\varepsilon >0$ there exists $\delta >0$ such that $|h_{\lambda'} - h_\lambda| \leq \varepsilon$ for $|\lambda - \lambda'| \leq \delta$, i.e., $(z,\lambda) \to h_\lambda(z)$ is continuous. 

Since $v\circ \phi(z,\lambda) \leq h_\lambda(z)$ for $v \in \textup{PSH}_g(E|_D)$, we get that also $u_g \circ \phi(z,\lambda) \leq h_\lambda(z),$ ultimately giving:
$$\limsup_{D \times \Bbb C^k \ni (z',\lambda') \to (z,\lambda) \in \partial D \times \Bbb C^k} u_g \circ \phi(z',\lambda') \leq g \circ \phi(z,\lambda).$$

This implies automatically that $u_g \in \textup{PSH}_g(E|_{D})$. As discussed before the proof, we immediately get that  $(i\partial \bar{\partial} u_g)^{m+r}=0$. 

If $\lambda \in \Bbb C^*$ and $v\in \textup{PSH}_g(E|_{D})$ then $v(z,\lambda \xi)/|\lambda| \in \textup{PSH}_g(E|_{D})$ due to homogeneity of $g$, it follows that $u_g(z,\lambda \xi)/|\lambda|=u_g(z,\xi)$, proving the homogeneity of $u_g$. Since $f_g\leq u_g$ and $f_g$ is a Finsler metric, $u_g(z,\xi)=0$ implies $\xi=0$; on the other hand, we know from above that $u_g \circ \phi(z,0) \leq h_0(z)$, but $h_0(z)$ is identically zero due to $h_0|_{\partial D} = g\circ \phi(\cdot,0)=0$, so $u_g(z,0)=0$. We have just shown that $u_g$ is a Finsler metric, i.e., $u_g\in F^\mathcal M_g$, implying that $u_g \leq f_g$, i.e. $u_g = f_g$. Since $f_g$ assumes the correct boundary values, so does $u_g$, and we obtain that $u_g$ solves \eqref{eq: Griff_extr_equation}, as desired.

Now we discuss uniqueness of $u_g$. Let $v \in F^\mathcal M_g$ be another solution to \eqref{eq: Griff_extr_equation}. Then naturally $v \leq f_g$ and $\varphi_v \leq \varphi_{f_g}$ with $\varphi_v \neq \varphi_{f_g}$. Since we have uniqueness of solutions to \eqref{eq: Giff_extremality_system} already established, we obtain that the measure $(\alpha + i\ddbar \varphi_v)^{m+r-1}$ puts mass on some open set $U \subset \Bbb P(E)|_D$. Since  $L(E)$ and $E$ biholomorphic away from the zero sections, and $v_L = h_L e^{\varphi_v}$, the lemma below then implies that $(i\ddbar v)^{m+r}$ also puts mass on an open subset of $E|_{D}$, a contradiction with the fact that $v$ solves \eqref{eq: Griff_extr_equation}.
\end{proof}

\begin{lemma} Suppose that $V \subset \Bbb C^{k}$ open and $U \subset V$ is a relatively compact open subset. Let $h \in \textup{PSH}(V) \cap L^\infty(V)$ be such that $\int_U(i\ddbar h)^k>0.$ Then we have that $e^{h(z)}|\xi| \in \textup{PSH}(V \times \Bbb C)$ and $\int_{U \times \{B(0,1) \setminus \overline{B(0,\frac{1}{2})}\}}(i\ddbar (e^{h(z)}|\xi|)^{k+1} > 0$. 
\end{lemma}

\begin{proof} Clearly $h(z) + \log|\xi|$ is psh on $V \times \Bbb C$, implying that so is $e^{h(z)}|\xi|$.

Let us assume momentarily that $h$ is smooth. Then on $V \times B(0,1) \setminus \overline{B(0,\frac{1}{2})}$ we get that 
$$i \ddbar (e^{h(z)}|\xi|)=e^{h(z)}|\xi| i \ddbar h +ie^h|\xi|\partial h\wedge \bar{\partial}h+ie^h\partial h\frac{\xi}{2|\xi|}d\Bar{\xi}-ie^h\Bar{\partial}h\frac{\Bar{\xi}}{2|\xi|}d\xi +\frac{e^{h(z)}}{4|\xi|}i \partial \xi \wedge \bar \partial \xi,$$ where the last four terms together represent a semipositive $(1,1)-$form. 
In particular, 
\begin{align*}
  i \ddbar (e^{h(z)}|\xi|)^{k+1} &\geq  e^{(k+1)h(z)}(|\xi| i \ddbar h)^k\wedge (i|\xi|\partial h\wedge \bar{\partial}h+i\partial h\frac{\xi}{2|\xi|}d\Bar{\xi}-i\Bar{\partial}h\frac{\Bar{\xi}}{2|\xi|}d\xi +\frac{1}{4|\xi|}i \partial \xi \wedge \bar \partial \xi)\\
  &=\frac{|\xi|^{k-1}}{4}e^{(k+1)h(z)} (i \ddbar h)^{k} \wedge (i \partial \xi \wedge \bar \partial \xi),  
\end{align*} by an argument on degrees.
Using the above estimate and Fubini's theorem we obtain that
$$\int_{U \times \{B(0,1) \setminus \overline{B(0,\frac{1}{2})}\}}(i\ddbar (e^{h(z)}|\xi|)^{k+1} \geq C \int_U(i\ddbar h)^k>0.$$
The case when $h$ is non-smooth follows using mollification of $h$ and standard convergence theorems of Bedford--Taylor theory. 
\end{proof}

\subsection{Degenerate CMAE with continuous boundary data}\label{subsec degenerate CMAE}

In this short subsection we give the simple argument of Theorem \ref{thm: deg_CMAE_cont_main}, building on estimates of Blocki \cite{Bl13} (see \cite{Ch00,Bo12,BD09} for closely related results):

\begin{theorem}\label{thm: deg_CMAE_cont_boundary_data} Given $v \in C(\partial{D} \times X)$ such that $v_z=v(z,\cdot) \in \textup{PSH}(X,\omega), \ z \in \partial D$, the Dirichlet problem \eqref{eq: DR_equation} has a unique solution $u  \in C(\overline{D} \times X) \cap \textup{PSH}(D \times X,\pi^* \omega)$. 
\end{theorem}

\begin{proof} To address existence, we first construct a sequence of approximate smooth subsolutions. Since $v \in C({\partial D} \times X)$ and $v_z \in \textup{PSH}(X,\omega), \ z \in \partial D$, we  can use Remark \ref{rem: split_holom_ex}(a) and Theorem \ref{thm: approx_family} to find $v^k  \in C^\infty({\partial D} \times X)$ such that $\omega + i\ddbar v^k_z>0, \ z \in \partial D$ and $v^k \searrow v$ uniformly. 

Using appropriate cutoffs and strong pseudoconvexity of $D$, we extend $v^k$ to 
$v^k  \in C^\infty({\overline{D}} \times X)$ such that $\pi^* \omega + i\ddbar v^k >0$ on ${\overline{D}} \times X$. Now let $u^k$ be the solution to the following PDE:
\begin{equation}
    \begin{cases}\label{eq: DR_equation_approx}
    (\pi^*\omega +i\partial \bar{\partial}u^k)^{n+m}=0 \text{  on $D\times X$},\\
    u^k|_{\partial D}=v^k\\
    \pi^*\omega +i\partial \bar{\partial}u^k\geq 0.
    \end{cases}
\end{equation} 
By \cite[Theorem 26]{Bl13} and an argument identical to Lemma \ref{lem: Lipschitz_solution}, we actually have that $u^k$ exists and is Lipschitz up to the boundary.

As in the last step of the proof of Theorem \ref{thm: extremal_PDE}, the comparison principle gives that $\{u^k\}_k$ is a Cauchy sequence of continuous functions, because so is the boundary data $\{v^k\}_k$. As a result $u^k \searrow u$ uniformly, with $u$ being continuous and equal to $v$ on the boundary. Basic theorems of Bedford--Taylor theory also imply that $    (\pi^*\omega +i\partial \bar{\partial}u)^{n+m}=0$. Hence $u$ is a continuous solution to \eqref{eq: DR_equation}, and is unique due to \cite[Theorem 21]{Bl13}.
\end{proof}

\section{Griffiths negativity and quantization}\label{sec Griff neg and quantization}

First we show that our definition of the Fubini--Study map on $\mathcal M^*_k$ from \eqref{eq: FS_k_def_general} is compatible with the classical one from \eqref{eq: FS_k_usual}: 
\begin{lemma}\label{lem: KRW-formula} For any $G \in \mathcal H_k$ and $x \in X$ we have that 
\begin{equation}\label{eq: KRW-formula}
FS_k(G)(x) := \frac{1}{k}\log \sup_{s\in H^0(X,L^k),G(s)\leq 1}h^k(s,s)= \frac{2}{k} \log\big[G^* (\hat{s}^*_k(x))\big]=: FS^*_k(G^*)(x),
\end{equation}
where $s_k^*:X \to (L^k)^*$ is any discontinuous section satisfying $(h^*)^k(s_k^*(x),s_k^*(x))=1, \ x \in X$. Moreover, $\hat s^*_k: X \to H^0(X,L^k)^*$ is the pointwise evaluation map of $s_k^*$.
\end{lemma}

\begin{proof} We start by noticing that two different choices of $s_k^*$ differ only by a unimodular complex factor, hence $FS_k^*$ is independent of such a choice. 

Consequently, it is enough to verify \eqref{eq: KRW-formula} for a specific choice of $s_k^*$ at a fixed $x \in X$. Let us pick a non-vanishing section $s \in H^0(U,L^k)$ in a neighborhood of $U$ of $x$. In this neighborhood  our desired section $s^*_k \in \Gamma(U,(L^k)^*)$ will be defined by 
\begin{equation}\label{eq: s_k_def}
s^*_k:=\frac{h^k(\cdot,s)}{h^k(s,s)^{\frac{1}{2}}}.
\end{equation}
Comparing the definitions \eqref{eq: FS_k_usual} and \eqref{eq: FS_k_def_general} at $x$, we conclude that
\begin{flalign*}
FS_k(G)(x) &= \frac{1}{k}\log \sup_{\sigma\in H^0(X,L^k),G(\sigma)\leq 1}h^k(\sigma,\sigma)(x)\\
&= \frac{1}{k}\log \sup_{\sigma\in H^0(X,L^k),G(\sigma)\leq 1} \bigg|\frac{\sigma(x)}{s(x)}\bigg|^2 h^k(s,s)\\
&= \frac{2}{k}\log \sup_{\sigma\in H^0(X,L^k),G(\sigma)\leq 1}|\hat{s}^*_k(x)(\sigma)|=\frac{2}{k} \log\big[G^* (\hat{s}^*_h(x))\big]=FS^*_k(G^*)(x),
\end{flalign*}
what we desired to prove.
\end{proof}

In the next lemma we show that $FS_k^*(\Lambda)$ is indeed $\omega$-psh for any $\Lambda \in \mathcal M^*_k$.

\begin{lemma}\label{lem: FS_k_PSH} For $\Lambda \in \mathcal M^*_k$ we have that $FS_k^*(\Lambda) \in \textup{PSH}(X,\omega)$.
\end{lemma}
\begin{proof} We need to show that $\omega+i\partial\bar{\partial}FS_k^*(\Lambda) \geq 0$. Pick $x \in X$, for the same choice of $s_k^* \in \Gamma(U,(L^k)^*)$ as in \eqref{eq: s_k_def}, it is enough to show this inequality on $U$. 

Notice that with this choice,
$$\hat{s}^*_k(x)(\sigma) = \frac{\sigma(x)}{s(x)} \cdot {h^k(s,s)^{\frac{1}{2}}}, \ \ x \in U.$$
Using \eqref{eq: FS_k_def_general} on $U$ 
$$\omega|_U+i\partial\bar{\partial}FS_k^*(\Lambda)|_U= -\frac{i}{k}\ddbar \log h^k(s,s) + i\partial\bar{\partial}FS_k^*(\Lambda)|_U=\frac{2i}{k}\partial\bar{\partial}\log \big[ \Lambda(\hat{s}_k^*(x)h^k(s,s)^{-\frac{1}{2}}) \big].$$
This last quantity is positive on $U$ by Proposition \ref{prop: Griff_Kob_negative_equiv} since $\hat{s}_k^*(x)h^k(s,s)^{-\frac{1}{2}}$ is holomorphic and $\Lambda \in \mathcal M^*_k$. 
\end{proof}

Recall that $D \ni z \to U_z \in \mathcal M_k^*$ is \emph{Griffiths negative} if $U$ is psh on $D  \times H^0(X,L^k)^*$. Given $v \in C(\partial D \times X)$ such that $v_z \in \textup{PSH}(X,\omega), \ z \in \partial D$, we consider the following families 
$$F_v^{\mathcal N,k}:= \{D \ni z \to U_z \in \mathcal N_k^* \textup{ is Griffiths negative and } \limsup_{z \to \partial D} U_z \leq H_k^*(v)\},$$
$$F_v^{\mathcal M,k}:= \{D \ni z \to U_z  \in \mathcal M_k^*\textup{ is Griffiths negative and } \limsup_{z \to \partial D} U_z \leq H_k^*(v)\},$$
where $H^*_k(\cdot)$ is the dual Hilbert map from \eqref{eq: Hilb_k_usual_dual}. Naturally we consider the following Perron type envelopes:
\begin{equation}\label{eq: Griff_extr_metric_norm_Perron_def}
U^{\mathcal M,k}:= \sup_{V \in F^{\mathcal M,k}_v} V, \ \ \ \ \ U^{\mathcal N,k}:= \sup_{V \in F^{\mathcal N,k}_v} V. \ \  
\end{equation}
From Theorems \ref{thm: extremal_PDE} and \ref{thm: main_PDE_noncompact} it follows that $U^{\mathcal M,k}$ is continuous on the total space and assumes $H^*_k(v)$ on the boundary.

Since the members in $F_v^{\mathcal N,k}$ are norms in the fiber direction, so is the envelope $U^{\mathcal N,k}$. Therefore, for $z\in D$ and $s,t\in H^0(X,L^k)^*$, we have
$$U^{\mathcal N,k}_z(s+t)\leq \textup{usc}(U^{\mathcal N,k})_z(s)+\textup{usc}(U^{\mathcal N,k})_z(t),$$
and by taking $\limsup$ on this inequality we see that $\textup{usc}(U^{\mathcal N,k})$ is a norm in the fiber direction. Since   $U^{\mathcal N,k} \leq \textup{usc}(U^{\mathcal N,k}) \leq U^{\mathcal M,k}$ and $\textup{usc}(U^{\mathcal N,k})$ is a norm in the fiber direction, it follows automatically that $U^{\mathcal N,k} \in F_v^{\mathcal N,k}$. More precisely, we now prove that $U^{\mathcal N,k}$ assumes the correct boundary value as well, following closely ideas of Slodkowski and Coifman--Semmes \cite{Sl1,CS93}:

\begin{prop}\label{prop: bdry data}
Both envelopes $U^{\mathcal M,k}$ and $U^{\mathcal N,k}$ assume the boundary data $H^*_k(v)$.
\end{prop}
\begin{proof} As discussed previously, Theorem \ref{thm: main_PDE_noncompact} implies the result for $U^{\mathcal M,k}$. We argue the statement for $U^{\mathcal N,k}$. By the comments before the proposition, we have $\limsup_{z\to \partial D} U^{\mathcal N,k}_z\leq H^*_k(v)$. So, for a fixed $\xi_0\in \partial D$, it is enough to show that 
\begin{equation}\label{eq: liminf_desired}
\liminf_{z\to \xi_0}  U^{\mathcal N,k}_z \geq H^*_k(v_{\xi_0}).  
\end{equation}
For $\xi \in \partial D$, we introduce 
$$w_0(\xi):=\inf_{H^0(X,L^k)^* \ni s\neq 0}\frac{H^*_k(v_{\xi})(s,s)^{\frac{1}{2}}}{H^*_k(v_{\xi_0})(s,s)^{\frac{1}{2}}}.$$
We claim that $w_0 \in C(\partial D)$; this is a special case of \cite[Lemma 11.5]{CS93} and we will be brief on the argument. Let us denote $$H^*_k(v_{\xi})(s,s)^{\frac{1}{2}}$$ by $p_\xi (s)$ where $\xi \in \partial D$ and $s\in H^0(X,L^k)^*$. Notice that for $\xi\in \partial D$, $p_\xi (\cdot)$ is a norm on $H^0(X,L^k)^*$. Meanwhile, for fixed $s$, $\partial D \ni \xi \mapsto p_\xi(s) $ is continuous because $v\in C(\partial D\times X)$ and the dual metric of a continuous Hermitian metric is also continuous. Such $p_\xi(s)$ is called a continuous norm-valued function. Let $s_1,...,s_l$ be a basis for $H^0(X,L^k)^*$. Identify $H^0(X,L^k)^*$ with $\mathbb{C}^l$ and let $\|\cdot\|$ be the Euclidean norm, namely, $$\| \sum_{i=1}^l a_i s_i   \| = \big(  \sum_{i=1}^l |a_i|^2       \big)^{\frac{1}{2}} \textup{ where $a_i\in \mathbb{C}$ for $i=1\sim l$.  }   $$ 
Using the fact $p_\xi(s)$ is a continuous norm-valued function, one can deduce the following inequalities. For a compact set $K\subset \partial D$, there exists a constant $C(K)$ such that $$p_\xi (s)\leq C(K) \|s\| \textup{ for $\xi\in K $ and $s\in H^0(X,L^k)^*$.    }   $$
Moreover, 
\begin{equation}\label{ineq 1}
  |  p_\xi(s)-p_\xi(s') |\leq p_\xi (s-s')\leq C(K)\| s-s' \|    \textup{ for $\xi\in K $ and $s,s'\in H^0(X,L^k)^*$.    }
\end{equation}
Fixing a metric $d$ on $\partial D$, we deduce from inequality (\ref{ineq 1}): given a compact set $K\subset \partial D$ and $\varepsilon>0$, there exists $\delta>0$ such that for $\xi_1,\xi_2\in K$ and $d(\xi_1,\xi_2)<\delta$, we have 
\begin{equation}\label{ineq 2}
    |p_{\xi_1}(s)-p_{\xi_2}(s)   |<\varepsilon\|s\|.
\end{equation}
From (\ref{ineq 1}) and (\ref{ineq 2}), $p_\xi(s)$ is continuous in $(\xi,s)$ jointly. By this joint continuity and the fact $p_\xi(\cdot)$ is a norm, we get: for a compact set $K\subset \partial D$, there exists $a>0$ such that
\begin{equation}\label{ineq 3}
p_\xi(s) \geq a\|s\|   \text{ for $\xi\in K$.  }
\end{equation}
We also notice that for $\xi\in \partial D$, the map $s\mapsto p_\xi(s)/p_{\xi_0}(s)$ is continuous on $\| s \|=1$, so
\begin{equation}\label{ineq 4}
  \inf_{s\neq 0 }\frac{p_\xi(s)}{p_{\xi_0}(s)}=\min_{\|s\|=1}\frac{p_\xi(s)}{p_{\xi_0}(s)}. 
\end{equation}
Using (\ref{ineq 2}), (\ref{ineq 3}), and (\ref{ineq 4}), the map 
$$ \partial D \ni \xi \mapsto  \inf_{s\neq 0 }\frac{p_\xi(s)}{p_{\xi_0}(s)}=w_0(\xi)  $$
is continuous as claimed. During the proof, one sees that $w_0(\xi)>0$ because the $\inf$ in defining $w_0$ is attained.

Let $w\in \textup{PSH}(D)\cap C(\overline{D})$ be such that $w= \log w_0$ on $\partial D$ (such a $w$ can be found as a solution to a complex Monge--Amp\`ere equation  
\begin{equation*}
    \begin{cases}
(i\partial \bar{\partial } w)^m=0\\
w|_{\partial D}=\log w_0
    \end{cases}
\end{equation*}
since $D$ is  strongly pseudoconvex). As $\xi_0\in \partial D$ is fixed, $H_k^*(v_{\xi_0})^{1/2}$ is a fixed norm on $H^0(X,L^k)^*$. We consider the Finsler norm
$$D \ni z \mapsto e^{w(z)}  H_k^*(v_{\xi_0})^{\frac{1}{2}}\in \mathcal{N}_k^*;      $$
this Finsler norm is Griffiths negative due to Proposition \ref{prop: Griff_Kob_negative_equiv} and the fact the map 
$$ D\times H^0(X,L^k)^* \ni (z,s) \mapsto w(z)+\log H_k^*(v_{\xi_0})(s,s)^{\frac{1}{2}}    $$ 
is psh. Furthermore, with the inequality
 $$ w_0(\xi)H^*_k(v_{\xi_0})(s,s)^{\frac{1}{2}}\leq H^*_k(v_{\xi})(s,s)^{\frac{1}{2}}, \ \xi \in \partial D,$$ we see that the Finsler norm $e^{w(z)}  H_k^*(v_{\xi_0})^{\frac{1}{2}}$ is in $F_v^{\mathcal N,k}$. As a result, $e^{w(z)}H^*_k(v_{\xi_0})(s,s)^{\frac{1}{2}}\leq U^{\mathcal N,k}_z(s)$ for $ z \in D$ since $U^{\mathcal N,k}$ is a Perron envelope. Letting $z\to \xi_0$ in this last estimate, we obtain \eqref{eq: liminf_desired}, as desired.
\end{proof}

\begin{remark}
Although both envelopes $U^{\mathcal M,k}$ and $U^{\mathcal N,k}$ have the same boundary values according to the above result, an example of Slodkowski \cite[Corollary 6.8]{Sl2} shows that the two envelopes are in general not the same!
\end{remark}
There exists a connection between Griffiths negativity and plurisubharmonicity in the quantum formalism, with the argument almost identical to the one in Lemma \ref{lem: FS_k_PSH}:
\begin{prop}\label{prop: semi_classic_max_princ} If $D \ni z \to V_z \in \mathcal M^*_k$ is Griffiths negative then $FS_k^*(V_z) \in \textup{PSH}(D \times X, \pi^* \omega)$.
\end{prop}

\begin{proof} We need to show that $\pi^* \omega+i\partial\bar{\partial}FS_k^*(V_z) \geq 0$. Pick $x \in X$, for the same choice of $s_k^* \in \Gamma(U,(L^k)^*)$ as in \eqref{eq: s_k_def}, it is enough to show this inequality on $D \times U$. 
Again, with this choice we have that
$$\hat{s}^*_k(x)(\sigma) = \frac{\sigma(x)}{s(x)} \cdot {h^k(s,s)^{\frac{1}{2}}}, \ \ x \in U.$$
Using \eqref{eq: FS_k_def_general} for $x \in U$  we have that
\begin{flalign*}
\pi^* \omega|_{D \times U}+i\partial\bar{\partial}FS_k^*(V_z)|_{D \times U} &= -\frac{i}{k}\ddbar \log h^k(s,s) + i\partial\bar{\partial}FS_k^*(V_z)|_{D \times U}\\
&=\frac{2i}{k}\partial\bar{\partial}\log \big[ V_z(\hat{s}_k^*(x)h^k(s,s)^{-\frac{1}{2}}) \big].
\end{flalign*}
This last quantity is positive on $D \times U$, as a consequence of holomorphicity of $\hat{s}_k^*(x)h^k(s,s)^{-\frac{1}{2}}$, that $z \to V_z$ is Griffiths negative, and Proposition \ref{prop: Griff_Kob_negative_equiv}. 
\end{proof}

We will need to recall the maximum principle due to Berndtsson, and a twisted version of it, that we will use. For a positive line bundle $(E,g)\to X$, we consider $\mathcal H_{E\otimes K_X}$, the space of positive Hermitian forms on $H^0(X,E\otimes K_X)$. Let $\eta =\Theta(g)$, we define a variant of Hilbert map $\text{Hilb}_{E\otimes K_X}: \mathcal H_{\eta} \to \mathcal H_{E\otimes K_X}$ by $$\text{Hilb}_{E\otimes K_X}(v)(s,s)=\int_X g(s,s)e^{-v}.$$

\begin{prop}\label{prop: bo}
If $v\in \textup{PSH}(D\times X, \pi^*\eta) \cap L^\infty(D \times X)$, then $D\ni z\mapsto \textup{Hilb}_{E\otimes K_X}(v_z)^*$ is Griffiths negative.
\end{prop}

\begin{proof}
Since Griffiths negativity is a local condition, we can assume that $D\subset \Bbb C^m$. We view  $z\mapsto \textup{Hilb}_{E\otimes K_X}(v_z)$ as a metric on the trivial bundle $D\times H^0(X,E\otimes K_X) \to D$. To apply Berndtsson's result \cite[Theorem 1.2]{Brn09} we first need to approximate $v$ via \cite[Theorem 2]{BK07}. By this last result, after possibly further shrinking $D$, there exist $\varepsilon_j \searrow 0$ and $v^j\in \textup{PSH}(D\times X, \pi^*\eta+\varepsilon_j(\omega_{D}+\pi^*\eta))\cap C^\infty(D\times X)$ decreasing to $v$ in $D\times X$, where $\omega_{D}=i\ddbar |z|^2$ for $z\in D$. Without loss of generality, we can futher assume $v,v_j \leq 0$, therefore $$u^j:=\frac{v^j}{1+\varepsilon_j}+\frac{\varepsilon_j}{1+\varepsilon_j}|z|^2 \in \textup{PSH}(D\times X, \pi^*\eta)\cap C^\infty(D\times X),$$
and $u_j \searrow v$.  More importantly, according to \cite[Theorem 1.2]{Brn09}, $z\mapsto \textup{Hilb}_{E\otimes K_X}(u^j_z)$ is Griffiths positive, so the dual is Griffiths negative \cite[Chapter VII]{De12}, i.e., $\textup{Hilb}_{E\otimes K_X}(u^j_z)^*$ is psh on $D\times H^0(X,E\otimes K_X)$. Since the dual Hilbert map is monotone increasing, we get that $\textup{Hilb}_{E \otimes K_X}(v_z)^*$ is psh on $D\times H^0(X,E \otimes K_X)$, as desired.
\end{proof}

Now, we replace $(E,g)$ by $(L^k\otimes K^*_X,h^k\otimes \omega^n)$ which is positive for large $k$ since $\Theta(h^k\otimes \omega^n)=k\omega+\Ric \ \omega$ (for convenience, we will denote this latter (1,1)-form by $\eta_k$). Also, we note that for $u\in \textup{PSH}(X,\eta_k/k) \cap \textup{PSH}(X,\omega) \cap L^\infty(X)$, we have $H_k(u)=\text{Hilb}_{L^k}(ku)$. 

\begin{prop}\label{prop: quant_max_princ} Suppose that $u \in \textup{PSH}(D \times X, \pi^* \omega)\cap L^\infty(D \times X)$ and $\pi^* \omega + i\ddbar u \geq \varepsilon \pi^* \omega$ on $D \times X$ for some $\varepsilon >0$. Then there exists $k_0(\varepsilon)$ such that for all $k \geq k_0$ we have that $D \ni z \to H_k^*(u_z) \in \mathcal H^*_k$ is Griffiths negative. 
\end{prop}
Note that Berndtsson works with the line bundles $L^k \otimes K_X$ in \cite{Brn09}, whereas in the above result we are dealing with $L^k$. As we will see, the condition $\pi^* \omega + i\ddbar u \geq \varepsilon \pi^* \omega$ is meant to remedy this discrepancy.

\begin{proof}
We have to show that, for large $k$, $ku_z\in \textup{PSH}(X,\eta_k)$ for $z\in D$, and $ku\in \text{PSH}(D\times X,\pi^* \eta_{k})$. 
To argue to first inclusion, $\eta_k+i\partial\bar{\partial}ku_z=k\omega +\Ric \omega +i\partial\bar{\partial}ku_z\geq k\varepsilon\omega+\Ric\omega$ which is positive for $k\geq k_0(\varepsilon)$. Similarly, to argue the second inclusion, we notice that 
$\pi^*\eta_k+i\partial\bar{\partial}ku \geq k\varepsilon\pi^*\omega+\pi^*\Ric \omega$ which is positive for $k\geq k_0(\varepsilon)$. As a result, by the discussion preceding the proposition and  Proposition \ref{prop: bo}, the map $z\mapsto H^*_k(u_z)=\text{Hilb}_{L^k}(ku_z)^*$ is Griffiths negative.
\end{proof}
For a usc function $f$ on $X$, we introduce $ P(f):=\sup\{h \in \textup{PSH}(X,\omega) \textup{ s.t. } h \leq  f\} \in \textup{PSH}(X,\omega)$ (see \cite{Brm13},\cite[Theorem A.7]{Da19}), which will be used in the proof of the next theorem.
\begin{theorem}\label{thm: C_0_quantization} Let $v \in C(\partial{D} \times X)$ such that $v_z \in \textup{PSH}(X,\omega), \ z \in \partial D$. The following hold:\\
\noindent (i) $\|FS_k^*(U^{\mathcal N,k}) - u \|_{C^0(D \times X)} \to 0$ as $k \to \infty$,\\
\noindent (ii) $\|FS_k^*(U^{\mathcal M,k}) - u \|_{C^0(D \times X)} \to 0$ as $k \to \infty$,\\
where $u$ is the solution to \eqref{eq: DR_equation}, and $U^{\mathcal N,k}/U^{\mathcal M,k}$ are the envelopes of Griffiths negative norms/metrics from \eqref{eq: Griff_extr_metric_norm_Perron_def}.
\end{theorem}

\begin{proof} The proof of (i) and (ii) is exactly the same, so we only argue (i).
Without loss of generality we can assume that $v \leq 0$. We fix $\delta > 1$ momentarily. Given $z \in \partial D$, we denote $v^\delta_z := P(\delta v_z)\in \textup{PSH}(X,\omega)$. Lemma \ref{lem: P_cont} below implies that $(z,x) \to v^\delta(z,x):=v_z^\delta(x)$ is continuous. 

Now let $u^\delta \in \textup{PSH}(D \times X, \pi^* \omega)$ be the solution of \eqref{eq: DR_equation} with boundary data $v^\delta$. For elementary reasons, we have that $\pi^*\omega + i \partial \bar \partial \frac{1}{\delta} u^{\delta} \geq \frac{\delta -1}{\delta} \pi^* \omega$. In particular, Proposition \ref{prop: quant_max_princ} can be used to conclude that $z \to H_k^*(\frac{1}{\delta} u^\delta_z)$ is Griffiths negative for $k \geq k_0(\delta)$. By Proposition \ref{prop: bdry data} we obtain that $H_k^*(\frac{1}{\delta} u^\delta_z) \leq \tilde U^k_z, \ z \in \overline{D}$, where $z \to \tilde U^k_z$ is the Griffiths extremal norm assuming the boundary values $H_k^*(\frac{1}{\delta} v^\delta)$. Since $\frac{1}{\delta} v^\delta \leq v$, due to monotonicity of $H^*_k$, we obtain the following sequence of inequalities:
\begin{equation}\label{eq: quant_max_est}
H_k^*\Big(\frac{1}{\delta} u^\delta_z\Big) \leq \tilde U^k_z \leq U^k_z, \ z \in D. 
\end{equation}
By Lemma \ref{lem: FS_k_H_k_convergence} below, there exist $C:=C(X,\omega)>0$ and $k_0(\delta,X,\omega)$ such that, for $k\geq k_0$,  $\frac{1}{\delta} u^\delta_z- \frac{C}{k}  \leq FS^*_k \circ H_k^*\Big(\frac{1}{\delta} u^\delta_z\Big)$, which together with (\ref{eq: quant_max_est}) gives $$ u^\delta_z - \frac{C}{k} \leq \frac{1}{\delta} u^\delta_z- \frac{C}{k} \leq  FS^*_k (U^k_z),$$ where in the first estimate we used that $u^\delta \leq 0$, which is a consequence of $v^\delta \leq 0$, since $v \leq 0$. Notice that  $v_z+(\delta-1)\inf_{\partial D\times X}v$ is a candidate for $v^\delta_z = P(\delta v_z)$ for any $z \in \partial D$, so for $k\geq k_0(\delta, X,\omega)$ we have that
\begin{equation}\label{eq:lower bound}
  u_z+(\delta-1)\inf_{\partial D\times X}v- \frac{C}{k}\leq u^\delta_z- \frac{C}{k}\leq FS^*_k (U^k_z).   
\end{equation}
 
 On the other hand, $FS_k^*(U^k_z) = FS_k^* \circ  H_k^*(v_z)\leq v_z+ C \frac{\log(k)}{k} + C M_{v_z}\big(\frac{1}{k}\big)$ on $\partial D \times X$, where the second inequality follows from Lemma \ref{lem: FS_k_H_k_convergence} below.   Proposition \ref{prop: semi_classic_max_princ} now gives that  \begin{equation}\label{eq: semi_class_max_es}
FS_k^*(U^k_z) \leq u_z+ C \frac{\log(k)}{k} + C M_v\Big(\frac{1}{k}\Big), \ z \in D. 
\end{equation}
Combining (\ref{eq:lower bound}) and (\ref{eq: semi_class_max_es}), we get the desired uniform convergence. 
\end{proof}

\begin{lemma}\label{lem: P_cont} Suppose that $v \in C(\partial D \times X)$. Then $\partial D \times X \in (z,x)\mapsto P(v_z)(x)$ is continuous.
\end{lemma}

\begin{proof} For $\chi \in C^2(X)$, it is well known that $P(\chi) \in C^{1,\bar 1}(X)$ \cite{Brm13} (see \cite[Theorem A.7]{Da19} for a self contained survey). Using uniform approximation by monotone smooth functions, one concludes that $P(\chi) \in C(X)$ if $\chi \in C(X)$. This implies that $P(v_z) \in C(X)$ for all $z \in \partial D$. 

Let $(z,x) \in \partial D \times X$ and let $\varepsilon > 0$ arbitary. For $\delta >0$ small enough we have that $|v_z - v_{z'}| \leq \varepsilon$ for $|z-z'| \leq \delta$. This immediately implies that $|P(v_z) - P(v_{z'})| \leq \varepsilon$. Also, after possibly shrinking $\delta$ (in a coordinate neighborhood of $x$) we have that $|x - x'| \leq \delta$ implies $|P(v)(z,x)-P(v)(z,x')| \leq \varepsilon$, due to continuity of $P(v_z)$. Putting everything together we obtain that
\begin{flalign*}
|P(v)(z,x) - P(v)(z',x')| \leq |P(v)(z,x) - P(v)(z,x')| + |P(v)(z,x') - P(v)(z',x')|\leq 2 \varepsilon,
\end{flalign*}
proving continuity of $P(v)$.
\end{proof}

Given $v \in C(X)$, let $M_v(r)= \sup\{|v(x)-v(y)|, \ x,y \in X \textup{ s.t. } d(x,y) \leq r\}$ be the ``modulus of continuity" for $v$, where $d(\cdot,\cdot)$ is the Riemannian distance associated to $\omega$.

\begin{lemma}\label{lem: FS_k_H_k_convergence} Suppose that $v \in \textup{PSH}(X,\omega) \cap C(X)$. Then there exists $C=C(X,\omega) >0$ and $k_0=k_0(X,\omega)$ such that 
$$FS_k \circ H_k(v) \leq v + C \frac{\log(k)}{k} + C M_v\Big(\frac{1}{k}\Big), \ \ k \geq k_0.$$
If in addition $\omega_v \geq \delta \omega$ for some $\delta >0$, then there exists $k_0=k_0(X,\omega,\delta)$ such that 
$$v - \frac{C}{k} \leq FS_k \circ H_k(v), \ \ k \geq k_0.$$
\end{lemma}
\begin{proof} 
For the first inequality, we adapt the argument of \cite[Proposition 4.2(iii)]{DLR18}. We fix $s \in H^0(X,L^k)$.  
We pick  $x \in X$, a coordinate neighborhood $B(x,\frac{2}{k})$, and a trivialization for $L$ on $B(x,\frac{2}{k})$. For $k \geq k_0(X,\omega)$, this is certainly possible near all $x \in X$. 
Using Cauchy's estimate we can start writing:
$$|s(x)|^2 \leq C k^{2n}\int_{B(x,\frac{1}{k})} |s(z)|^2,$$
for some absolute constant $C$. 
On $B(x,\frac{2}{k_0})$ we denote our Hermitian metric $h = e^{-\varphi}$ for some $\varphi \in C^\infty(B(x,\frac{2}{k_0}))$, and we note that there exists $C= C(X,\omega)$ such that $\sup_{B(x,\frac{1}{k})} \varphi - \varphi(x) \leq \frac{C}{k}, \ x \in X.$ Indeed, $C$ only depends on the Lipschitz constant of $\varphi$. 

Using the above estimate we can continue:
\begin{flalign*}
h^k(s(x),s(x))&=|s(x)|^2 e^{-k\varphi(x)} \leq C k^{2n} \frac{ e^{\sup_{B(x,\frac{1}{k})}k\varphi}}{ e^{k\varphi(x)}}\int_{B(x,\frac{1}{k})} h^k(s,s) \\
&\leq C k^{2n} e^{k \big(\sup_{B(x,\frac{1}{k})} v+ \sup_{B(x,\frac{1}{k})} \varphi - \varphi(x)\big)} \int_{X} h^k(s,s)e^{-kv}\omega^n. \\
&\leq C k^{2n} e^{k \big(v(x) + C M_v(\frac{1}{k}) + \frac{C}{k}\big)} \int_{X} h^k(s,s)e^{-kv}\omega^n.
\end{flalign*}
Consequently, the definition \eqref{eq: FS_k_usual} implies that $FS_k \circ H_k (v)(x) \leq v(x) + C \frac{\log k}{k} + C M_v\big(\frac{1}{k}\big)$, what we desired prove.

The second inequality is just a direct consequence of the Ohsawa--Takegoshi extension theorem \cite{OT87}. Indeed, fixing $x \in X$, by the version of this result in \cite[Theorem 2.11]{DLR18}, we have that for all $k \geq k_0(X,\omega,\delta)$ there exists $s \in H^0(X,L^k)$ such that $\int_X h^k(s,s) e^{-kv} \omega^n \leq C h^k(s,s)(x) e^{-kv(x)}.$ Naturally,  using \eqref{eq: FS_k_usual}, 
this implies that
$$v(x) \leq \frac{1}{k} \log \frac{h^k(s,s)}{\int_X h^k(s,s) e^{-kv} \omega^n}+ \frac{C}{k} \leq FS_k \circ H_k(v) + \frac{C}{k}.$$

\end{proof}

\footnotesize
\let\OLDthebibliography\thebibliography % squeezes Bibliography
\renewcommand\thebibliography[1]{
  \OLDthebibliography{#1}
  \setlength{\parskip}{1pt}
  \setlength{\itemsep}{1pt}
}

\normalsize
\noindent{\sc University of Maryland}\\
{\tt tdarvas@math.umd.edu}\vspace{0.1in}\\
\noindent{\sc Purdue University}\\
{\tt wu739@purdue.edu}


\begin{thebibliography}{00}
\footnotesize
\bibitem{Brn09} B. Berndtsson, Curvature of vector bundles associated to holomorphic fibrations,  Ann. of Math. (2) 169 (2009), no. 2, 531--560.
\bibitem{Brn09b} B. Berndtsson, Probability measures related to geodesics in the space of K\"ahler metrics, arXiv:0907.1806.
\bibitem{BCKR16} B. Berndtsson, D. Cordero-Erausquin, B. Klartag, Y.A. Rubinstein, Complex interpolation of $\Bbb R$-norms, duality and foliations, arXiv:1607.06306, to appear in J. Eur. Math. Soc.
\bibitem{BD09} R. Berman, J.P. Demailly,  Regularity of plurisubharmonic upper envelopes in big cohomology classes. Arxiv, 0905.1246, Perspectives in analysis, geometry and topology, Progr. Math. 296, Birkh\"auser/Springer, New York, 2012, 39--66.
\bibitem{BK12} R. Berman, J. Keller, About Bergman geodesics and homogenous complex Monge--Amp\`ere equations, Lecture Notes in Mathematics 2038, Springer, 283--302 (2012). 
\bibitem{Brm13} R. Berman, From Monge--Amp\`ere equations to envelopes and geodesic rays in the zero temperature limit, Math. Z. 291 (2019), no. 1-2, 365--394.
\bibitem{BH99} M. Bridson, A. Haefliger, Metric spaces of non-positive curvature. Grundlehren der Mathematischen Wissenschaften, 319. Springer--Verlag, Berlin, 1999.
\bibitem{BT76} E.~Bedford, B.A.~Taylor, \href{http://projecteuclid.org/euclid.bams/1183537618}{The Dirichlet problem for a complex Monge-Amp\`ere equation},  Invent. Math. 37  (1976), no. 1, 1--44. 
\bibitem{Bo12} S. Boucksom, Monge--Amp\`ere equations on complex manifolds with boundary,  Complex Monge--Amp\`ere equations and geodesics in the space of K\"ahler metrics, V. Guedj editor. Lecture Notes in Mathematics, 2038. Springer, Heidelberg (2012).
\bibitem{BK07} Z. Blocki, S. Kolodziej, On regularization of plurisubharmonic functions on manifolds. Proc. Amer. Math. Soc. 135 (2007), no. 7, 2089--2093. 
\bibitem{Bl11} Z. Blocki, A gradient estimate in the Calabi--Yau theorem, Mathematische Annalen 344 (2009), 317--327.
\bibitem{Bl13} Z. Blocki, The complex Monge--Amp\`ere equation in K\"ahler geometry, course given at CIME Summer School in Pluripotential Theory, Cetraro, Italy, July 2011, eds. F. Bracci, J. E. Fornaess, Lecture Notes in Mathematics 2075, pp. 95--142, Springer, 2013. 
\bibitem{Ca97} D. Catlin, The Bergman kernel and a theorem of Tian, in: Analysis and geometry in several
complex variables (Katata, 1997), Trends Math., Birkh\"auser, 1999, pp. 1--23.
\bibitem{Ch00} X.X. Chen, The space of K\"ahler metrics, J. Differential Geom. 56 (2000), no. 2,
189--234.
\bibitem{CS12} X.X. Chen, S. Sun, Space of K\"ahler metrics (V)– K\"ahler quantization, in: Metric and Differential
Geometry (X.-Z. Dai et al., Eds.), Springer, 2012, pp. 19--42.
\bibitem{CCRSW82} R.R. Coifman, M.Cwikel, R. Rochberg, Y. Sagher, G. Weiss, A theory of complex interpolation for families of Banach spaces, Advan. in Math.,43(1982), pp. 203-229.
\bibitem{CS93}    R.R. Coifman, S. Semmes, Interpolation  of  Banach  spaces,  Perron  processes,  and  Yang--Mills,  American  Journal  of  Mathematics 115(1993),  no.  2, 243--278.
\bibitem{Da17} T. Darvas, The Mabuchi completion of the space of K\"ahler potentials, Amer. J. Math. 139 (2017), no. 5, 1275-1313.
\bibitem{Da19} T. Darvas, Geometric pluripotential theory on K\"ahler manifolds, arXiv:1902.01982.
\bibitem{DLR18} T. Darvas, C.H. Lu, Y.A. Rubinstein, Quantization in geometric pluripotential theory,  arXiv:1806.03800, to appear in Comm. Pure Appl. Math.
\bibitem{DR16} T. Darvas, Y.A. Rubinstein, Kiselman's principle, the Dirichlet problem for the Monge-Amp\`ere equation, and rooftop obstacle problems, J. Math. Soc. Japan 68 (2016), no. 2, 773--796.
\bibitem{De94} J. P. Demailly, Regularization of closed positive currents of type $(1,1)$ by the flow of a Chern connection. Contributions to complex analysis and analytic geometry.
Aspects Math., E26, Vieweg, Braunschweig, 1994, 105--126.
\bibitem{De12} J.P. Demailly, Complex Analytic and Differential Geometry, textbook available at the website of the author.
\bibitem{Do92} S. K. Donaldson, Boundary value problems for Yang-Mills fields, J. Geom. Phys. 8 (1992), no. 1-4, 89--122.
\bibitem{Do99} S. K. Donaldson. Symmetric spaces, K\"ahler geometry and Hamiltonian dynamics. In Northern California Symplectic Geometry Seminar, vol. 196 of Amer. Math. Soc. Transl. Ser. 2, 13--33. Amer. Math. Soc., Providence, RI, 1999.
\bibitem{Do01} S. K. Donaldson, Scalar curvature and projective embeddings, I, J. Differential Geom. 59 (2001), 479--522.
\bibitem{FLW09} H. Feng, K. Liu, X. Wan, Complex Finsler vector bundles with positive Kobayashi curvature, arXiv:1811.08617.
\bibitem{GT01}  D. Gilbarg, N.S. Trudinger, Elliptic partial differential equations of second order. Classics in Mathematics. Springer--Verlag, Berlin, 2001.
\bibitem{Gr69} P. A. Griffiths: Hermitian differential geometry, Chern classes and positive vector bundles, Global Analysis, papers in honor of K. Kodaira, Princeton Univ. Press, Princeton(1969), 181--251.
\bibitem{GZ17}V. Guedj, A. Zeriahi, Degenerate complex Monge--Amp\`ere equations. EMS Tracts in Mathematics, 26. European Mathematical Society (EMS), Z\"urich, 2017. xxiv+472 pp.
\bibitem{Gu90} R.C. Gunning, Introduction to holomorphic functions of several variables. Vol. I. Function theory. The Wadsworth \& Brooks/Cole Mathematics Series. Wadsworth \& Brooks/Cole Advanced Books \& Software, Pacific Grove,  1990.
\bibitem{Ho90} L. H\"ormander, An introduction to complex analysis in several variables. Third edition. North-Holland Mathematical Library, 7. North--Holland Publishing Co., Amsterdam, 1990.
\bibitem{Ko75} S. Kobayashi, Negative vector bundles and complex Finsler structures. Nagoya Mathematical Journal 57 (1975).
\bibitem{Ko96} S. Kobayashi, Complex Finsler vector bundles. Finsler geometry. Contemp. Math., vol. 196. Amer. Math. Soc., 1996, 145--153.
\bibitem{Ko14} S. Kobayashi, Differential geometry of complex vector bundles. Vol. 793. Princeton University Press, 2014.
\bibitem{Lu00} Z. Lu, On the lower order terms of the asymptotic expansion of Tian--Yau--Zelditch. Amer. J. Math.
122 (2000).
\bibitem{Ma87}T. Mabuchi, Some symplectic geometry on compact K\"ahler manifolds I, Osaka J. Math. 24, 1987, 227--252.
\bibitem{MT07} C. Mourougane, S. Takayama, Hodge metrics and positivity of direct images, J. Reine Angew. Math. 606(2007).
\bibitem{OT87} T. Ohsawa, K. Takegoshi, On the extension of $L^2$ holomorphic functions, Math. Z. 195 (1987), 197--204.
\bibitem{PS06} D.H. Phong, J. Sturm, The Monge--Amp\`ere operator and geodesics in the space of K\"ahler potentials, Invent. Math. 166 (2006), 125–149
\bibitem{RZ10} Y.A. Rubinstein, S. Zelditch,  Bergman approximations of harmonic maps into the space of K\"ahler metrics on toric varieties. J. Symplectic Geom. 8 (2010), no. 3, 239--265.
\bibitem{Ro84} R. Rochberg, Interpolation of Banach spaces and negatively curved vector bundles. Pacific J. Math. 110 (1984), no. 2, 355--376.
\bibitem{Se92} S. Semmes, Complex Monge--Amp\`ere and symplectic manifolds, Amer. J. Math. 114(1992), 495--550.
\bibitem{Sl1} Z. Slodkowski, Complex interpolation of normed and quasinormed spaces in several dimensions. I, Trans. Amer. Math. Soc. 308 (1988), 685--711. 
\bibitem{Sl2} Z. Slodkowski, Complex interpolation of normed and quasinormed spaces in several dimensions. II. Properties of harmonic interpolation,  Trans. Amer. Math. Soc. 317 (1990), 255--285. 
\bibitem{Sl3}Z. Slodkowski, Complex interpolation for normed and quasi-normed spaces in several dimensions. III. Regularity results for harmonic interpolation, Trans. Amer. Math. Soc. 321 (1990), 305--332.
\bibitem{Sl4} Z. Slodkowski, Polynomial hulls with convex fibers and complex geodesics.
J. Funct. Anal. 94 (1990), no. 1, 156--176.
\bibitem{SZ10}  J. Song, S. Zelditch, Bergman metrics and geodesics in the space of K\"ahler metrics on toric varieties,
Anal. PDE 3 (2010), 295--358.
\bibitem{Ti88} G. Tian, K\"ahler metrics on algebraic manifolds, Ph.D. Thesis, Harvard University, 1988.
\bibitem{Ti90} G. Tian, On a set of polarized K\"ahler metrics on algebraic manifolds. J. Differential Geom., 32
(1990), 99–130 
\bibitem{Ya87} S.-T. Yau, Nonlinear analysis in geometry, Enseign. Math. 33 (1987), 109--158.
\bibitem{Wo07} Wong, Pit-Mann. A survey of complex Finsler geometry. Finsler geometry, Sapporo 2005—in memory of Makoto Matsumoto, 375--433, Adv. Stud. Pure Math., 48, Math. Soc. Japan, Tokyo, 2007.
\bibitem{Ze98} S. Zelditch, Szeg\H o kernel and a theorem of Tian, Int. Math. Res. Notices 6 (1998), 317--331.\vspace{0.1in}
\end{thebibliography}
\end{document}